\documentclass[envcountsame,envcountresetchap,graybox]{svmult}

\usepackage{mathptmx}       
\usepackage{helvet}         
\usepackage{courier}        
\usepackage{makeidx}         
\usepackage{graphicx}        
\usepackage{multicol}        
\usepackage[bottom]{footmisc}

\makeindex  

\smartqed

\input xy
\xyoption{all}
\usepackage{eucal}

\usepackage{amscd}
\usepackage{amssymb}
\usepackage{amsmath}

\usepackage{color}

\spnewtheorem*{thmA}{Theorem}{\bf}{\itshape}
\spnewtheorem*{thmB}{Theorem}{\bf}{\itshape}

\newcommand{\rH}{{\rm H}}
\newcommand{\rK}{{\rm K}}
\newcommand{\ab}{{\rm ab}}
\newcommand{\bT}{\mathbb T}
\newcommand{\bP}{\mathbb P}
\newcommand{\bN}{\mathbb N}
\newcommand{\bF}{\mathbb F}

\newcommand{\Gm}{\G_{\rm m}}

\newcommand{\w}{\overrightarrow{0w}}
\def\01{{\overrightarrow{01}}}
\def\10{{\overrightarrow{10}}}

\newcommand{\surj}{\twoheadrightarrow} 
\newcommand{\inj}{\hookrightarrow}

\newcommand{\AJ}{\alpha}

\DeclareMathOperator{\inv}{{\rm inv}}

\newcommand{\Z}{\mathbb{Z}}
\newcommand{\Zhat}{\hat{\mathbb{Z}}}

\newcommand{\Q}{\mathbb{Q}}
\newcommand{\Qbar}{\overline{\mathbb{Q}}}
\newcommand{\kbar}{\overline{k}}
\newcommand{\R}{\mathbb{R}}
\newcommand{\C}{\mathbb{C}}
\newcommand{\G}{\mathbb{G}}
\newcommand{\frakf}{\mathfrak{f}}
\newcommand{\proj}{\mathbb P}
\newcommand{\pmQ}{\mathbb P^1_{\mathbb{Q}} - \{0,1,\infty \}}

\newcommand{\pmk}{\mathbb P^1_{k} - \{0,1,\infty \}}
\newcommand{\pmkbar}{\mathbb P^1_{\overline{k}} - \{0,1,\infty \}}

\newcommand{\Jac}{\operatorname{Jac}}

\newcommand{\Gal}{\operatorname{Gal}}
\newcommand{\mdl}{\operatorname{mod}}
\newcommand{\Br}{\operatorname{Br}}

\newcommand{\Spec}{\operatorname{Spec}}

\newcommand{\Ker}{\operatorname{Ker}}

\newcommand{\hidden}[1]{\footnote{Hidden:  #1}}
\renewcommand{\hidden}[1]{}

\begin{document}

\setcounter{chapter}{11}
\title{On $3$-nilpotent obstructions to $\pi_1$ sections for $\pmQ$}
\titlerunning{On $3$-nilpotent obstructions to $\pi_1$ sections} 

\author{Kirsten Wickelgren\thanks{Supported by an NSF Graduate Research Fellowship, a Stanford Graduate Fellowship, and an American Institute of Math Five Year Fellowship}}

\institute{Kirsten Wickelgren \at Harvard University, Cambridge MA USA \\
\email{wickelgren@post.harvard.edu}}


\maketitle


\abstract*{We study which rational points of the Jacobian of $\proj^1_K -\{0,1,\infty\}$ can be lifted to sections of geometrically $3$-nilpotent quotients of \'etale $\pi_1$ over the absolute Galois group. This is equivalent to evaluating certain triple Massey products of elements of $K^\ast \subseteq \rH^1(G_K, \Zhat(1))$. 
For $K=\Q_p$ or $\R$, we give a complete mod $2$ calculation. This permits some mod $2$ calculations for $K = \Q$. These are computations of obstructions of Jordan Ellenberg.}

\abstract{We study which rational points of the Jacobian of $\proj^1_k -\{0,1,\infty\}$ can be lifted to sections of geometrically $3$-nilpotent quotients of \'etale $\pi_1$ over the absolute Galois group. This is equivalent to evaluating certain triple Massey products of elements of $k^\ast \subseteq \rH^1(G_k, \Zhat(1))$ or $\rH^1(G_k, \Z/2)$. For $k=\Q_p$ or $\R$, we give a complete mod $2$ calculation. This permits some mod $2$ calculations for $k = \Q$. These are computations of obstructions of Jordan Ellenberg.}








\section{Introduction} 

The generalized Jacobian of a pointed smooth curve can be viewed as its abelian approximation. It is natural to consider non-abelian nilpotent approximations. Gro\-then\-dieck's anabelian conjectures predict that smooth hyperbolic curves over certain fields are controlled by their \'etale fundamental groups. In particular, approximating $\pi_1$ should be similar to approximating the curve. We study the effect of $2$ and $3$-nilpotent quotients of the \'etale fundamental group of $\proj^1 - \{0,1,\infty\}$ on its rational points, using obstructions of Jordan Ellenberg. 

More specifically, a pointed smooth curve $X$ embeds into its generalized Jacobian via the Abel-Jacobi map. Applying $\pi_1$ to the Abel-Jacobi map gives the abelianization of the \'etale fundamental group of $X$. Quotients by subgroups in the lower central series lie between $\pi_1(X)$ and its abelianization, giving rise to obstructions to a rational point of the Jacobian lying in the image of the Abel-Jacobi map. These obstructions were defined by Ellenberg in \cite{Ellenberg_2_nil_quot_pi_better}. 

For simplicity, first assume that $X$ is a proper, smooth, geometrically connected curve over a field $k$. The absolute Galois group of $k$ will be denoted by $$G_k = \Gal(\kbar/k)$$ where $\kbar$ denotes an algebraic closure of $k$. Assume that $X$ is equipped with a $k$ point, denoted $b$ and used as a base point; a $k$-variety will be said to be \textit{pointed} if it is equipped with a $k$-point.  The point $b$ gives rise to an Abel-Jacobi map $$\AJ: X \rightarrow \Jac X $$ from $X$ to its Jacobian, sending $b$ to the identity, and applying $\pi_1$ to $\AJ \otimes \kbar$ produces the abelianization of $\pi_1(X_{\kbar})$. For any pointed variety $Z$ over $k$, there is a natural map $$\kappa: Z(k) \rightarrow \rH^1(G_k, \pi_1(Z_{\kbar}))$$ where $Z(k)$ denotes the $k$ points of $Z$. In particular, we have the commutative diagram \begin{equation}\label{XJacH1cd} \xymatrix{ \Jac(X)(k) \ar[r] & \rH^1(G_k, \pi_1(X_{\kbar})^{\ab}) \\
X(k) \ar[u] \ar[r] & \rH^1(G_k, \pi_1(X_{\kbar})) \ar[u]}\end{equation} Any $k$ point of $\Jac(X)$ which is in the image of the Abel-Jacobi map satisfies the condition that its associated element of $\rH^1(G_k, \pi_1(X_{\kbar})^{\ab})$ lifts through the map \begin{equation}\label{H1pitoab} \rH^1(G_k, \pi_1(X_{\kbar})) \rightarrow \rH^1(G_k, \pi_1(X_{\kbar})^{\ab}).\end{equation} Therefore showing that the associated element of $\rH^1(G_k, \pi_1(X_{\kbar})^{\ab})$ does not admit such a lift obstructs this point of the Jacobian from lying on the curve. Ellenberg's obstructions are obstructions to lifting through the map (\ref{H1pitoab}). Since they obstruct a conjugacy class of sections of $\pi_1$ of $\Jac X \rightarrow \Spec k$ from being the image of a conjugacy class of sections of $\pi_1$ of $X \rightarrow \Spec k$, they are being called ``obstructions to $\pi_1$ sections" in the title.  They arise from the lower central series and are defined in~\ref{secpi1pmksec}.

More specifically, Ellenberg's obstruction $\delta_n$ is the $\rH^1 \rightarrow \rH^2$ boundary map in $G_k$ cohomology for the extension \begin{equation}\label{pinn+1extninto}1 \rightarrow [\pi]_n/[\pi]_{n+1} \rightarrow \pi/[\pi]_{n+1} \rightarrow \pi/[\pi]_{n} \rightarrow 1\end{equation} where $\pi$ is the \'etale fundamental group of $X_{\kbar}$, and
$$
\pi = [\pi]_1 \supset [\pi]_2 \supset [\pi]_3 \supset \ldots
$$ 
denotes the lower central series of $\pi$. The obstruction $\delta_n$ is regarded as a multi-valued function on $\rH^1(G_k, \pi^{\ab})$ via $\rH^1(G_k,\pi/[\pi]_{n}) \rightarrow \rH^1(G_k, \pi^{\ab})$ and also on $\Jac X (k)$ via $\kappa$. 

Now assume that $X= \proj_k^1 -\{0,1,\infty\}$ and that $k$ is a subfield of $\C$ or a completion of a number field. \hidden{In the latter case, we will also fix the associated decomposition group.} By replacing the Jacobian by the generalized Jacobian and enlarging $X(k)$ to include $k$ rational tangential base points, we obtain a commutative diagram generalizing \eqref{XJacH1cd}. The same obstructions to lifting through (\ref{H1pitoab}) define obstructions $\delta_n$ for $X$. As there is an isomorphism $$\pi \cong \langle x, y \rangle^{\wedge}$$ between $\pi$ and the profinite completion of the topological fundamental group of $\proj^1_{\C}- \{0,1, \infty \}$, bases of $$[\pi]_n/[\pi]_{n+1} \cong \Zhat(n)^{N(n)}$$ can be specified by order $n$ commutators of $x$ and $y$, decomposing the obstructions $\delta_n$ into multi-valued, partially defined maps $$ \rH^1 (G_k, \Zhat(1) \oplus \Zhat(1)) \dashrightarrow \rH^2(G_k, \Zhat(n)).$$ 

Section~\ref{delta23cohops} expresses $\delta_2$ and $\delta_3$ in terms of cup products and Massey products. For $a$ in $k^\ast$, let $\{a\}$ denote the image of $a$  in $\rH^1(G_k, \Zhat(1))$ under the Kummer map. For $(b,a)$ in $\Jac X (k) \cong (\G_m \times \G_m)(k)$, the obstruction $\delta_2$ is given \cite{Ellenberg_2_nil_quot_pi_better} by $\delta_2(b,a) = \{b\} \cup \{a\}$. It is a charming observation of Jordan Ellenberg that this computation shows that the cup product factors through $\rK_2(k)$ (Remark \ref{delta2K2sec}). The obstruction $\delta_3$ is computed by Theorem \ref{delta_3_Massey_prod_8_10} $$ \delta_{3,[[x,y],x]} (b,a) =  \langle \{-b\}, \{b\}, \{a \} \rangle$$ $$\delta_{3,[[x,y],y]} (b,a) = - \langle \{-a\}, \{a\},\{ b\} \rangle - f \cup \{a\},$$ where $f \in \rH^1(G_k, \Zhat(2))$ is associated to the monodromy between $0$ and $1$. The indeterminacy of the Massey product and the conditions required for its definition coincide with the multiple values assumed by $\delta_3$ and the condition for its definition. 

Section \ref{sectiondelta_2} contains computations of $\delta_2$, and its mod $2$ reduction. In particular, \ref{2nilseccon} provides points on which to evaluate $\delta_3$, which can be phrased as the failure of a $2$-nilpotent section conjecture for $\pmk$. Tate's computation of $\rK_2(\Q)$ gives a finite algorithm for determining whether or not $\delta_2 (b,a) = 0$ for $k=\Q$ described in \ref{delta2Qalg}.

The main results of this paper are in Section \ref{delta3sec}. The mod $2$ reduction of $\delta_3$ for a finite extension $k_v$ of $\Q_p$ with $p$ odd is computed:  

\smallskip
{\bf Theorem \ref{Prop_ploc_mod2_3nil_obs}} {\em Suppose that $\delta_2^{\mdl 2} (b,a) = 0$. Then $\delta_3^{\mdl 2} (b,a) \neq 0$  if and only if one of the following holds: \begin{itemize}
\item  $\{ -b \} = 0$ and $\{2 \sqrt{-b}\} \cup \{a\} \neq 0$. 
\item   $\{ -a \} = 0$ and $\{2 \sqrt{-a}\} \cup \{b\} + \{2\} \cup \{a\} \neq 0$. 
\item $\{b\} = \{a\}$ and $\{ 2 \sqrt{b} \sqrt{a} \} \cup \{a\} \neq 0$. 
\end{itemize}}
\noindent where equalities such as $\{ -b \} = 0$ take place in $\rH^1(G_{k_v}, \Z/2)$ and non-equalities such as  $\{2 \sqrt{-b}\} \cup a \neq 0$ take place in $\rH^2(G_{k_v}, \Z/2)$. The cocycle $$f: G_k \rightarrow [\pi]_2/([\pi]_3([\pi]_2)^2) \cong \Z/2$$ described above is known due to contributions of Anderson, Coleman, Deligne, Ihara, Kaneko, and Yukinari, and $f$'s computation is inputed to Theorem \ref{Prop_ploc_mod2_3nil_obs}. For points $$(b,a) \in (\Z - \{0\}) \times (\Z - \{0\}) \subset \Jac X (\Q_p)$$ such that $p$ divides $ab$ exactly once, the vanishing of $\delta_3^{\mdl 2}$ for $\Q_p$ can be expressed in terms of the congruence conditions

 \begin{itemize}
 \item  $\delta_2^{\mdl 2} (b,a) = 0 \iff a+b$ is a square mod $p$ 
 \item When $\delta_2^{\mdl 2} (b,a) = 0$, $\delta_3^{\mdl 2} (b,a) = 0 \iff a+b$ is a fourth power mod $p$
 \end{itemize}
 
\noindent This is Corollary~\ref{local_mod2_delta2_3_cor_version}. As the image of $X$ in its Jacobian consists of $(b,a)$ such that $b + a = 1$, and the image of the tangential points of $X$ are $(b,a)$ such that $b + a = 0$, or $b=1$, or $a=1$, we see in Corollary~\ref{local_mod2_delta2_3_cor_version} that $\delta_3^{\mdl 2}$ vanishes on the points and tangential points of $X$. Of course, $\delta_3^{\mdl 2}$ is constructed to satisfy this property, but here it is visible that $\delta_2^{\mdl 2}$ and $\delta_3^{\mdl 2}$ are increasingly accurate approximations to $X$ inside its Jacobian.

The obstruction $\delta_3^{\mdl 2}$ for $k = \R$ is computed in \ref{delta_3mod2R}. Consider $k=\Q$. Although an element of $\rH^2(G_{\Q}, \Z/2)$ is $0$ if and only if its restriction to all places, or all but one place, vanishes, the previous local calculations can only be combined to produce a global calculation when the Massey products are evaluated locally using compatible defining systems. This involves the local-global comparison map on Galois cohomology with coefficients in a $2$-nilpotent group. See Remark \ref{Poitou-Tate_remark}. In \ref{loc_global_delta3mod2(-p^3,p)}, such lifts are arranged and the local calculations are used to show that $$\delta_3^{\mdl 2} (-p^3,p)=0$$ for $k=\Q$. Proposition \ref{eval_delta_3_2_spec_lift} computes $\delta_3^{\mdl 2}$  on a specific lift of $(-p^3,p)$ for $k=\Q$, which is equivalent to the calculation of the $G_{\Q}$ Massey products with $\Z/2$ coefficients $\langle \{p^3 \}, \{-p^3 \}, \{p\} \rangle$ and $\langle \{-p\}, \{p\}, \{-p^3\} \rangle$ for any specified defining system.

{\bf Acknowledgments: } I wish to thank Gunnar Carlsson, Jordan Ellenberg, and Mike Hopkins for many useful discussions. I thank the referee for correcting sign errors in Proposition \ref{8/10_delta_3_cocycle_form} and \ref{2nilquotf}. I also thank Jakob Stix for clarifying \ref{2nilseccon}, shortening the proofs of Lemma \ref{achoose2_when_-a=0}, Propositions \ref{Prop_delta_3_2R} and \ref{delta3mod2(-p^3,p)}, and for extensive and thoughtful editing. 

\section{Ellenberg's obstructions to $\pi_1$ sections for $\pmk$}\label{secpi1pmksec}

We work with fields $k$ which are subfields of $\C$ or completions of a number field. In the latter case, fix an embedding of the number field into $\C$, as well as an algebraic closure $\kbar$ of $k$, and an embedding  $\overline{\Q} \subset \overline{k}$, where $\overline{\Q}$ denotes the algebraic closure of $\Q$ in $\C$. These specifications serve to choose maps between topological and \'etale fundamental groups, as in \eqref{pixywedeiso}.

This section defines Ellenberg's obstructions. In \ref{Tangential_base_pts}, we recall Deligne's notion of a tangential point \cite[\S 15]{Deligne} \cite{Nakamura}, and define in \eqref{eq:kappadef} the map $\kappa$ from $k$ points and tangential points to $\rH^1(G_k, \pi_1(X_{\kbar}))$. We then specialize to $X=\pmk$, give the computation of $\kappa$ composed with $$\rH^1(G_k, \pi_1(X_{\kbar})) \rightarrow \rH^1(G_k, \pi_1(X_{\kbar})^\ab)$$ in \eqref{aj(x)_is_(x,1-x)} and Lemma \ref{AbelJacobi(tgtpnt)}, and define Ellenberg's obstructions in \ref{Ellenbergob_subsection}. 

\subsection{Tangential base points, path torsors, and the Galois action}\label{Tangential_base_pts}

Let $X$ be a smooth, geometrically connected curve over $k$ with smooth compactification $X \subseteq \overline{X}$ and $x \in \overline{X}(k)$, so in particular, $x$ could be in $(\overline{X}-X)(k)$.

A local parameter $z$ at $x$ gives rise to an isomorphism $\widehat{\mathcal{O}}_{\overline{X},x} \xleftarrow{\cong} k[[z]]$, where $\widehat{\mathcal{O}}_{\overline{X},x}$ denotes the completion of the local ring of $x$. 
 
Let $\kbar$ be a fixed algebraic closure of $k$. Since we assume that $k$ has characteristic $0$, the field of Puiseux series \[
\kbar((z^\Q)) : = \cup_{n\in \Z_{>0}} \kbar ((z^{1/n}))
\]
is algebraically closed. The composition 
\[
\Spec \kbar((z^\Q))  \rightarrow \Spec k[[z]] \cong \Spec \widehat{\mathcal{O}}_{\overline{X},x} \rightarrow \overline{X} 
\]
factors through the generic point of $\overline{X}$ and thus defines a geometric point of $X$
$$b_z: \Spec  \kbar((z^\Q)) \rightarrow X$$ 
that will be called the {\em tangential base point} of $X$ at $x$ in the direction of $z$. The tangential base point $b_z$  determines an embedding 
$$ 
k(X) \subset k((z)) \subset \kbar((z^\Q)) .
$$ 

The coefficientwise action of $G_k = \Gal(\kbar/k)$ on $\cup_{n\in \Z_{>0}} \kbar ((z^{1/n}))$ gives a splitting of $G_{k((z))} \rightarrow G_k$. Combined with the embedding $k(X) \subset k((z))$, this splitting gives a splitting of $G_{k(X)} \rightarrow G_k$ and therefore a splitting of 
\[
\pi_1^{et}(X, b_z) \rightarrow G_k,
\]
see \cite[V Prop 8.2]{sga1}, and a $G_k$ action on $\pi_1^{et}(X_{\kbar}, b_z)$. 

\label{X(k)_to_H1} A \textit{geometric point associated to a $k$ point or tangential point} will mean $$x: \Spec \Omega_x \rightarrow X$$ where $\Omega_x$ is an algebraically closed extension of $\kbar$, such that either $x$ arises as a tangential base point as described above or $x$ has a $k$ point as its image. Such a geometric point determines a canonical geometric point of $X_{\kbar}$, and the associated fiber functor has a canonical $G_k$ action. A \text{path} between two such geometric points $b$ and $x$ is a natural transformation of the associated fiber functors, and the set of paths\[
\pi_1(X_{\kbar}; b,x)
\]
from $b$ to $x$ form a trivial $\pi_1^{et}(X_{\kbar},b)$ torsor whose $G_k$ action determines an element  
\[
[\pi_1(X_{\kbar}; b,x)] \in \rH^1(G_k,\pi_1^{et}(X_{\kbar},b))
\]
represented by the cocycle 
\begin{equation}\label{htpy_fixed_pt_cocycle} 
g \mapsto \gamma^{-1} \circ g(\gamma) 
\end{equation} where $\gamma$ is any path from $b$ to $x$. Composition of paths is written right to left so that $\gamma^{-1} \circ g(\gamma) $ is the path formed by first traversing $g(\gamma)$ and then $\gamma^{-1}$.

A local parameter $z$ at a point $x$ of $\overline{X}$ determines a tangent vector 
\[
\Spec k[[z]]/\langle z^2\rangle \rightarrow \overline{X}.
\] 
For $b$ or $x$ a tangental base point, the associated element of $\rH^1(G_k,\pi_1^{et}(X_{\kbar},b))$ only depends on the choice of local parameter up to the associated tangent vector. 
Furthermore, if $x$ is a $k$ tangential point which comes from a tangent vector at a point $p$ of $X$, then 
\[
[\pi_1(X_{\kbar}; b,x)]  = [\pi_1(X_{\kbar}; b,p)].
\]
This describes a map
\begin{equation} \label{eq:kappadef} 
\kappa = \kappa_{(X,b)} \ : \ X(k) \cup \bigcup_{x \in \overline{X}-X}(T_x \overline{X}(k)-\{0\}) \to \rH^1(G_k, \pi_1(X_{\kbar},b)).
\end{equation}

\begin{example}\label{kappa_Gm_example} 
\label{Kummer}  \label{Kummermodncocycle} \label{Kummermod2cocycle} 
\label{addingKummercocycles} 

The boundary map for the Kummer sequence 
\begin{equation}\label{KummerSESmodn}
1 \to \Z/n\Z(1) \to \Gm \xrightarrow{n\cdot} \Gm \to 1
\end{equation}
yields in the limit over all $n$ the Kummer map 
\begin{equation}\label{eq:Kummermap}
\ k^\ast \to \rH^1(G_k,\Zhat(1))
\end{equation}
which is represented on the level of cocycles by 
\begin{equation} \label{eq:Kummercocycle}
\sigma \mapsto \{z\}(\sigma) = \Big(\sigma(\sqrt[n]{z}) / \sqrt[n]{z}\Big)_n
\end{equation}
for any compatible choice of $n^{\rm th}$ roots of $z \in \kbar^\ast$. This cocycle, or by abuse of notation also the class it represents, will also be denoted by $z$, or denoted by $\{z\}$ if there is possible confusion. 

When $n=2$, both choices of square root of $z$ produce the same cocycle. Furthermore, canonically $\mu_2 = \Z/2\Z$ and thus we have a well-defined homomorphism 
$$ 
k^\ast \rightarrow C^1(G_k, \Z/2\Z),
$$
with $C^1(G_k, \Z/2\Z)$ the group of continuous $1$-cocycles of $G_k$ with values in $\Z/2\Z$.

For $(X,b) = (\G_m, 1)$, the map $\kappa$ is the Kummer map: for $$x \in \G_{{\rm m}, k} (k)=k^\ast,$$ choose compatible $n^{th}$ roots $\sqrt[n]{x}$ of $x$, and choose $1$ as the $n^{th}$ root of unity for each $n \in \Z_{>0}$.  These choices determine a path $\gamma$ from $1$ to $x$ as follows.  On the degree $n$ cover 
$$
p_n \ : \ \G_{{\rm m},\kbar} \rightarrow \G_{{\rm m},\kbar}
$$ 
given by $t \mapsto t^n$, the path $\gamma$ maps 
\[
\gamma \ : \ p_n^{-1}(1) \to p_n^{-1}(x)
\]
by multiplication by $\sqrt[n]{x}$. For $g \in G_k$, the path $g \gamma$ is the path sending $g1$ to $g(\sqrt[n]{x})$, thus multiplies by $g(\sqrt[n]{x})$. We conclude
\[
\kappa(x) = \gamma^{-1} \circ g(\gamma) = g (\sqrt[n]{x})/\sqrt[n]{x} = \{ x\}(g).
\] Identifying the choice of path from $1$ to $x$ with the choice of compatible $n^{\rm th}$ roots of $x$, there is an equality of cocycles $\kappa(x)= \{x \}$.

Similarly, for $w \in T_0 \proj^1_k (k)-\{0\} = k^\ast$, a compatible choice of $n^{th}$ roots $\sqrt[n]{w}$ of $w$ determines a path $\gamma$ from $1$ to $\w$, where $\w$ is the $k$ tangential point \begin{equation}\label{tgtptgmn} \Spec \kbar ((z^{\Q})) \rightarrow \Spec k[t,\frac{1}{t}]\end{equation} given by $t \mapsto  wz$, by defining $\gamma$ to map $\zeta \in \mu_n(\kbar)=p_n^{-1}(1)$ to the point of $p_n^{-1}(\w)$ given by \eqref{tgtptgmn} and $t \mapsto \sqrt[n]{w} \zeta z^{1/n}$.  For any $g \in G_k$, we have $g(\gamma) ( \zeta) =  g (\gamma(g^{-1} \zeta))$ is the path given by \eqref{tgtptgmn} and $t \mapsto (g\sqrt[n]{w}) \zeta z^{1/n}$, whence $\gamma( (g\sqrt[n]{w}) \zeta/\sqrt[n]{w}) = g(\gamma) ( \zeta)$. We conclude \[
\kappa(\w) = \gamma^{-1} \circ g(\gamma) = g (\sqrt[n]{w})/\sqrt[n]{w} = \{ w\}(g) = \kappa (w).
\] 
\end{example}

\begin{example}\label{kappa_Gm_example_tgtl_base_pt}
The map $\kappa_{(X,b)}$ depends on the choice of base point, even when $\pi_1(X_{\kbar}, b)$ is abelian and is therefore independent of $b$. In this case, if $b_1$ and $b_2$ are two geometric points associated to a $k$ point or tangential point of $X$, a straightforward cocycle manipulation shows that $$\kappa_{(X,b_2)}(x)= \kappa_{(X,b_1)}(x) - \kappa_{(X,b_1)}(b_2)$$ for any $k$ point or tangential point $x$.\hidden{Choose paths l12: b1-> b2

l2: b2 -> x
 
k_b1 (x)(g) = l12^{-1}l2^{-1} gl2 gl12 = l12^{-1} k_b2(x)(g) l^12 k_b1(b2)

Thus k_b1(x) - k_b1(b2) = k_b2(x)}  In particular, Example \ref{kappa_Gm_example} implies that $$\kappa_{(\G_m, \01)} =  \kappa_{(\G_m, 1)}.$$
\end{example}

$\kappa$ is often called a non-abelian Kummer map.

For higher dimensional geometrically connected varieties $Z$ over $k$, we will not need the notion of a tangential base point, but we will use the map 
\begin{equation} \label{eq:kappadefZ} 
\kappa = \kappa_{(Z,b)} \ : \ Z(k) \to \rH^1(G_k, \pi_1(Z_{\kbar},b))
\end{equation}
defined precisely as in the case of curves above. 

\label{Galoisactionpi} 

Let $X=\pmk= \Spec k[t,\frac{1}{t},\frac{1}{1-t}]$, and base $X$ at $\01$ as in Example~\ref{kappa_Gm_example} \eqref{tgtptgmn}. We fix an isomorphism \begin{equation}\label{pixywedeiso}\pi = \pi_1^{et}(\pmkbar, \01) \cong \langle x , y \rangle^{\wedge}\end{equation} between $\pi$ and the profinite completion of the free group on two generators as follows: recall that we assume that $k$ is a subfield of $\C$ or the completion of a number field at a place, and we have fixed $\C \supset \Qbar \subseteq \kbar$, see \ref{secpi1pmksec}. The morphisms $\C \supset \Qbar \subseteq \kbar$ and the Riemann existence theorem give an isomorphism $\pi \cong \pi_1^{top}(\proj^1_{\C} - \{0,1,\infty \}, \01)^{\wedge},$ where the base point for the topological fundamental group, also denoted $\01$, is the tangent vector at $0$ pointing towards $1$. Let $x$ be a small counterclockwise loop around $0$ based at $\01$. Let $y'$ be the pushforward of $x$ by automorphism of $\proj^1_{\C} - \{0,1,\infty \}$ given $1 \mapsto 1-t$, so in particular, $y'$ is a small loop around $1$ based at $\10$, where $\10$ is the tangent vector at $1$ pointing towards $0$. Conjugating $y'$ by the direct path along the real axis between $\10$ and $\01$ produces a loop $y$, and an isomorphism $\pi_1^{top}(\proj^1_{\C} - \{0,1,\infty \}, \01) = \langle x,y \rangle,$ giving \eqref{pixywedeiso}.

An element $\sigma \in G_k$ acts on $\pi$ by
\begin{eqnarray} \label{G_k_action_F_2} 
\sigma(x) & = & x^{\chi(\sigma)} \\
\sigma(y) & = & \frakf(\sigma)^{-1}y^{\chi(\sigma)}\frakf(\sigma) = [\frakf(\sigma)^{-1},y^{\chi(\sigma)}] y^{\chi(\sigma)}, \notag
\end{eqnarray} 
\noindent where $\mathfrak{f}: G_k \rightarrow [\pi]_2$ is a cocycle with values in  the commutator subgroup $[\pi]_2$ of $\pi$ (coming from the monodromy of the above path from $\01$ to $\10$), and $\chi: G_k \rightarrow \Zhat^*$ denotes the cyclotomic character, see \cite{Ihara_GT}. 

\subsection{The Abel-Jacobi map for $\pmk$}\label{AJpmk_to_GmGm} 

Let $X \subseteq \overline{X}$ denote a smooth curve over $k$ inside its smooth compactification. The generalized Jacobian $\Jac X$ of $X$, is the algebraic group of equivalence classes of degree $0$ divisors of $X$, where two divisors are considered equivalent if they differ by $\operatorname{Div}(\phi)$ for a rational function $\phi$ such that $\phi(p)=1$ for all $p$ in $\overline{X} - X$. It follows that $\Jac(X)$ is an extension of $\Jac (\overline{X})$ by the torus\[
\bT = \Big( \prod_{p \in \overline{X} - X} \operatorname{Res}_{k(p)/k} \G_{{\rm m}, k(p)} \Big)/ \G_{{\rm m},k}
\]
where $p$ ranges over the closed points of $\overline{X}-X$ with residue field $k(p)$, the torus 
$\operatorname{Res}_{k(p)/k} G_{m, k(p)}$ denotes the restriction of scalars of $\G_{{\rm m}, k(p)}$ to $k$, and where $\G_{{\rm m},k}$ acts diagonally. For more information on generalized Jacobians see \cite{alg_grps_class_fields}.

For $X= \pmk$, the Jacobian of $\overline{X} = \bP^1_k$ s trivial. The complement of $X$ in $\overline{X}$ consists of three rational points and
\begin{equation}\label{Jac=Gm2}
\Jac(\pmk) \cong \G_{{\rm m},k} \times \G_{{\rm m},k}.
\end{equation}
Since the fundamental group of a connected group is abelian, the fundamental group does not depend on base points. We find 
\[
\pi_1(\Jac(\pmkbar)) = \pi_1(\G_{{\rm m},\kbar},1) \times \pi_1(\G_{{\rm m},\kbar},1) = \Zhat(1) \oplus \Zhat(1).
\]
We choose an isomorphism \eqref{Jac=Gm2} by sending $\operatorname{Div}(f)$ for a rational function $f$ on $\proj^1$ to $f(0)/f(\infty) \times f(1)/f(\infty)$ in $\G_{\rm m} \times \G_{\rm m}$. 

The Abel-Jacobi map based at $\01$ 
\[
\AJ \ : \ \pmk \to \Jac(\pmk) = \G_{{\rm m},k} \times \G_{{\rm m},k}
\]
\[
t \mapsto (t,1-t)
\]
induces the abelianization 
\begin{equation}\label{piab=zhat12}
\pi = \pi_1(\pmkbar, \01) \surj \pi^\ab = \pi_1(\Jac(\pmkbar)) = \Zhat(1) \oplus \Zhat(1).
\end{equation}
\begin{remark}\label{Zhat1=Zhatchi}
From \eqref{pixywedeiso} and \eqref{G_k_action_F_2}, we have fixed an isomorphism $$\pi^{ab} \cong \Zhat(\chi)x \oplus \Zhat(\chi)y.$$ The isomorphism \eqref{piab=zhat12} above $\pi^{ab} \cong \Zhat(1) \oplus \Zhat(1)$ is the composition of the former with the isomorphism $$\Zhat(\chi) \cong \Zhat(1) := \varprojlim_n \mu_{n, \kbar}$$ corresponding to the compatible choice of roots of unity given by the action of the loop $x$ on the fiber over $\01$ on the finite \'etale covers of $ \G_{m,\kbar}$, i.e. to the choice $(\zeta_n)_{n \in \Z_{>0}}$, $\zeta_n = e^{2\pi i/n}$ of compatible primitive $n^{th}$ roots of unity. We will hereafter identify $$\Zhat(1)= \Zhat(\chi)$$ by this isomorphism, and for typographical reasons, we will use the notation $\Zhat(n)$ for $\Zhat(\chi^n)$, although the group law will be written additively. 
\end{remark}

By Example \ref{kappa_Gm_example}, the map $\kappa$ for the $k$ scheme $\Gm \times \Gm$ pointed by $(1,1)$ is  two copies of the Kummer map:\[
k^\ast \times k^\ast \rightarrow \rH^1(G_k, \Zhat(1)) \times \rH^1(G_k, \Zhat(1))
\]
\[
b \times a \mapsto \{b \} \times \{a\}.
\]

By functoriality of $\kappa$ and Example \ref{kappa_Gm_example_tgtl_base_pt}, the following diagram is commutative: \begin{equation}\label{aj(x)_is_(x,1-x)}
\xymatrix{H^1(G_k, \pi) \ar[rr] && H^1(G_k, \pi^{ab})\\
(\pmk)(k) \ar[u]^{\kappa_{\01}} \ar[rr]^{t \mapsto (t,1-t)} && (\G_m \times \G_m)(k) \ar[u]^{\kappa_{(1,1)}} }
\end{equation}



We can similarly compute the image of the $k$ tangential points of $\pmk$ in $\rH^1(G_k, \pi^{ab})$, which is what we now do (also see \cite{Ellenberg_2_nil_quot_pi_better}). We define a map 
\[
\alpha \ : \  \cup_{t = 0,1,\infty} (T_t \bP^1-\{0\})(k) \to (\G_m \times \G_m)(k) 
\]
by, for $w \in k^\ast$,
\[
\alpha(\w)  =   (w,1), \quad
\alpha(\overrightarrow{1w})  =  (1,-w), \quad
\alpha(\overrightarrow{\infty w})  =  (w^{-1},-w^{-1})
\] where $\overrightarrow{1w}$ is the pushforward of $\w$ under $t \mapsto t + 1$, and $\overrightarrow{\infty w}$ is the pushforward of $\w$ under $t \mapsto 1/t$.  

\begin{lemma} \label{AbelJacobi(tgtpnt)}
The following diagram commutes:
$$
\xymatrix@M+1ex@R-2ex{\rH^1(G_k, \pi) \ar[rr] &&\rH^1(G_k, \pi^{\ab})\\
 \cup_{t = 0,1,\infty} (T_t \bP^1-\{0\})(k) \ar[u] \ar[rr]^{\alpha} && (\G_m \times \G_m)(k) \ar[u] . } 
$$
\end{lemma}

\begin{proof}
The commutativity of the two diagrams
\[
\hbox{
\xymatrix{X \ar[rr]^{\alpha}  \ar[d]^{t \mapsto 1-t} && \Jac X \ar[d]^{(b,a) \mapsto (a,b)} \\ 
X \ar[rr]^{\alpha} && \Jac X } 
} \qquad 
\hbox{
\xymatrix{X \ar[rr]^{\alpha}  \ar[d]^{t \mapsto 1/t} && \Jac X \ar[d]^{(b,a) \mapsto (\frac{1}{b},\frac{-a}{b})} \\ 
X \ar[rr]^{\alpha} && \Jac X } 
}
\]
reduces the lemma for $t=1$ or $\infty$ (respectively) to the case $t=0$.

By functoriality of $\kappa$ applied to the Abel-Jacobi map $t \mapsto (t, 1-t)$, we have that the image of $\kappa_{(\pmk, \01)} (\w)$ in $\rH^1(G_k, \pi^{\ab})$ is $$\kappa_{(\Gm \times \Gm, \01 \times 1)}(\w, \overrightarrow{1(-w)}) = \kappa_{(\Gm,\01)}(\w) \times \kappa_{(\Gm,1)}(\overrightarrow{1(-w)}).$$ The geometric point $\overrightarrow{1(-w)}$ factors through $1-w \frac{\partial}{\partial t} : \Spec \overline{k} [[z]] \rightarrow \G_{{\rm m},\kbar}$ \[ t \mapsto 1 -w z.\] The quotient map $\Spec \kbar \rightarrow \Spec \kbar [[z]]$ given by $z \mapsto 0$ gives a bijection between the fiber over $1-w \frac{\partial}{\partial t}$ and the fiber over $1$ of the multiplication by $n$ cover $$p_n: \G_{{\rm m},\kbar} \rightarrow \G_{{\rm m},\kbar}.$$ These bijections determine a Galois equivariant path between $1$ and $\overrightarrow{1(-w)}$ showing that $$\kappa_{(\Gm,1)}(\overrightarrow{1(-w)}) = \kappa_{(\Gm,1)}(1) .$$ The lemma now follows from Examples \ref{kappa_Gm_example} and \ref{kappa_Gm_example_tgtl_base_pt}. 
\qed \end{proof}

\subsection{Ellenberg's obstructions}\label{Ellenbergob_subsection} Let $\pi$ be $\pi_1(\proj^1_{\overline{k}} - \{0,1,\infty \}, \01)$ or more generally $\pi$ can be any profinite group with a continuous $G_k$ action, e.g. the \'etale fundamental group of a $k$-variety after base change to $\overline{k}$.

The lower central series of $\pi$ is the filtration of closed characteristic subgroups $$ \pi= [\pi]_1 \supset [\pi]_2 \supset \ldots \supset [\pi]_n \supset \ldots $$ where the commutator is defined $[x,y] = x y x^{-1} y^{-1}$, and $[\pi]_{n+1} = \overline{[\pi, [\pi]_n]}$ is the closure of the subgroup generated by commutators of elements of $[\pi]_n$ with elements of $\pi$. 

The central extension $$1 \rightarrow [\pi]_n/[\pi]_{n+1} \rightarrow \pi/[\pi]_{n+1} \rightarrow \pi/[\pi]_{n} \rightarrow 1$$ gives rise to a boundary map in continuous group cohomology $$\delta_n: \rH^1(G_k,\pi/[\pi]_{n} ) \rightarrow \rH^2(G_k,[\pi]_n/[\pi]_{n+1} )$$ that is part of an exact sequence of pointed sets (see for instance \cite[I 5.7]{Serre_gal_coh}), \begin{align*}1 \rightarrow &([\pi]_n/[\pi]_{n+1})^{G_k}  \rightarrow (\pi/[\pi]_{n+1})^{G_k}  \rightarrow (\pi/[\pi]_{n})^{G_k} \\
 \rightarrow &\rH^1(G_k,[\pi]_n/[\pi]_{n+1})  \rightarrow \rH^1(G_k,\pi/[\pi]_{n+1}) \rightarrow \rH^1(G_k,\pi/[\pi]_{n}) \\
 \rightarrow &\rH^2(G_k,[\pi]_n/[\pi]_{n+1}).\end{align*}

The $\delta_n$ give a series of obstructions to an element of $\rH^1(G_k, \pi/[\pi]_2)$ being the image of an element of $ \rH^1(G_k, \pi)$, thereby also providing a series of obstructions to a rational point of the Jacobian coming from a rational point of the curve: to a given element $x$ of $\rH^1(G_k, \pi/[\pi]_2)$, if $\delta_2 (x) \neq 0$, then $x$ is not the image of an element of $ \rH^1(G_k, \pi)$. Otherwise, $x$ lifts to $\rH^1(G_K, \pi/[\pi]_3)$. Apply $\delta_3$ to all the lifts of $x$. If $\delta_3$ is never $0$, then $x$ is not the image of an element of $ \rH^1(G_k, \pi)$. Otherwise, $x$ lifts to $\rH^1(G_k, \pi/[\pi]_4)$, and so on. 
\begin{definition}
For $x$ in $\rH^1(G_k, \pi/[\pi]_2)$, say that $\delta_n x = 0$ if $x$ is in the image of \begin{equation}\label{Hn+1toH2def}\rH^1(G_k, \pi/[\pi]_{n+1}) \rightarrow \rH^1(G_k, \pi/[\pi]_{2}).\end{equation} Otherwise, say $\delta_n x \neq 0$.
\end{definition}
Let $X=\pmk$, or more generally $X$ could be a smooth, geometrically connected, pointed curve over $k$, with an Abel-Jacobi map $X \rightarrow \Jac X$. As we are interested in obstructing points of the Jacobian from lying on $X$, it is convenient to identify a rational point of $\Jac X$ with its image under $\kappa$ cf. \eqref{eq:kappadef}. 
\begin{definition}\label{def:deltan=0Jacpt}
For a $k$-point $x$ of $\Jac X$, say that $\delta_n x = 0$ if $\kappa_{(\Jac, 0)} x$ is in the image of \eqref{Hn+1toH2def}, where $0$ denotes the identity of $\Jac X$. Otherwise, say $\delta_n x \neq 0$.
\end{definition}

For $k$ a number field, and $K$ the completion of $k$ at a place $\nu$, the obstruction $\delta_n$ for $K$ will sometimes be denoted $\delta_n^{\nu}$, and applied to elements of $\rH^1(G_k, \pi/[\pi]_n)$; it is to be understood that one first restricts to $\rH^1(G_K, \pi/[\pi]_n)$. In other words, given $x$ in $\rH^1(G_k, \pi/[\pi]_2)$, the meaning of $\delta_n^{\nu} x = 0$ is that there exists $x_{n+1}$ in $\rH^1(G_K, \pi/[\pi]_{n+1})$ lifting the restriction of $x$ to $H^1(G_K, \pi/[\pi]_2)$. For $x$ a point of $\Jac X(k)$, the meaning of $\delta_n^{\nu} x = 0$ is that $\delta_n^{\nu} \kappa_{(\Jac X, 0)} x = 0$. This is equivalent to taking the image of $x$ under $\Jac X (k) \rightarrow \Jac X_K (K)$ and applying Definition \ref{def:deltan=0Jacpt} with $K$ as the base field.
  
\label{deltan2:def}Any filtration of $\pi$ by characteristic subgroups such that successive quotients give rise to central extensions produces an analogous sequence of obstructions. For instance, consider the lower exponent $2$ central series $$
\pi= [\pi]^2_1 \supset [\pi]^2_2 \supset \ldots \supset [\pi]^2_n \supset \ldots ,
$$ 
defined inductively by
\[
[\pi]^2_{n+1} = \overline{[\pi,[\pi]^2_n]\cdot ([\pi]^2_n)^2}
\] where $[\pi,[\pi]^2_n]\cdot ([\pi]^2_n)^2$ denotes the subgroup generated by the indicated commutators and the squares of elements of $[\pi]^2_n$. The resulting obstructions are denoted $\delta_n^2$, and will also be evaluated on $\Jac X(k)$ in the following manner: $\pi^\ab$ maps to $(\pi/[\pi]_{n+1}^2)^\ab$, giving a map $$\rH^1(G_K, \pi^{\ab}) \rightarrow \rH^1(G_K, (\pi/[\pi]_{n+1}^2)^{\ab}),$$ where either $K=k$ or $K$ is the completion of a number field $k \subset \C$ at a place $v$ as above. Precomposing with $\kappa$ for the Jacobian \eqref{eq:kappadef} gives a map $$\Jac(X)(k) \rightarrow H^1(G_K, (\pi/[\pi]_{n+1}^2)^{\ab}).$$ For $x$ in $\Jac(X)(k)$, say $\delta_n^2 x = 0$ (respectively $\delta_n^{(2,v)} x = 0$) if there exists $x_{n+1}$ in $\rH^1(G_K, \pi/[\pi]_{n+1}^2)$ such that $x$ and $x_{n+1}$ have equal image in $\rH^1(G_K, (\pi/[\pi]_{n+1}^2)^{ab})$. Otherwise, say $\delta_n^2 x \neq 0$ (respectively $\delta_n^{(2,v)} x \neq 0$).  The obstruction $\delta_n^{\nu}$ for $\nu$ the place $2$ will not be considered, so the notation $\delta_n^2$ will not be ambiguous. Obstructions $\delta_n^m$ corresponding to the lower exponent $m$ central series, $$[\pi]^m_{n+1} = \overline{[\pi,[\pi]^m_n]\cdot ([\pi]^m_n)^m}$$ are defined similarly.

As one final note of caution, $\rH^1(G_K,\pi/[\pi]_{n} )$ is in general only a pointed set. Furthermore, even for $n=2$, the map $\delta_2$ is not a homomorphism, see Proposition \ref{delta2cocycle}. 
 
\section{The obstructions $\delta_2$ and $\delta_3$ as cohomology operations}\label{delta23cohops} 

We express $\delta_2$ and $\delta_3$ for $\proj^1_{k} - \{0,1,\infty\}$ in terms of cohomology operations, where $k$ is a subfield of $\C$ or the completion of a number field at a place; in the latter case, fix an embedding of the number field into $\C$ and an embedding $\overline{\Q} \subset \overline{k}$, giving the isomorphism $\pi=\pi_1^{et}(\pmkbar, \01) \cong \langle x,y \rangle^{\wedge}$ of \eqref{pixywedeiso}. We will use the following notation:  

\label{defC*(G,A)} For elements $x$ and $y$ of a group, let $[x,y] = x y x^{-1} y^{-1} $ denote their commutator. For a profinite group $G$ and a profinite abelian group $A$ with a continuous action of $G$, let $(C^*(G,A), D)$ be the complex of inhomogeneous cochains of $G$ with coefficients in $A$ as in \cite[I.2~p. 14]{coh_num_fields}. For $c \in C^p(G, A)$ and $d \in C^q(G, A')$, where $A'$ is a profinite abelian group with a continuous action of $G$, let $c \cup d$ denote the cup product $c \cup d \in C^{p+q}(G, A \otimes A')$  $$(c \cup d)(g_1,\ldots,g_{p+q}) = c(g_1,\ldots,g_p) \otimes g_1\cdots g_pd(g_{p+1},\ldots, g_{p+q}),$$ which induces a well defined map on cohomology, and gives $C^*(G,A)$ the structure of a differential graded algebra, for $A$ a commutative ring, via $A \otimes A \rightarrow A$. For a profinite group $Q$, no longer assumed to be abelian, the continuous $1$-cocycles 
\[
Z^1(G_k,Q) = \{s: G_k \rightarrow Q | s \textrm{ is continuous, } s(g h) = s(g) g s(h) \}
\]
of $G_k$ with values in $Q$ form a subset of the set of continuous inhomogeneous cochains  
\[
C^1(G_k,Q) = \{s: G_k \rightarrow Q | s \textrm{ is continuous} \}.
\] See \cite[I \S5]{Serre_gal_coh} for instance. For $s \in C^1(G_k,Q)$, let $Ds: G_k \times G_k \rightarrow Q$ denote the function $Ds(g,h) = s(g) gs(h) s(gh)^{-1}$. 

\subsection{The obstruction $\delta_2$ as a cup product}  For any based curve $X$ over $k$ with fundamental group of $X_{\kbar}$ denoted $\pi_1$, \cite[Thm p 242]{Zarkhin} or \cite[Prop. 1]{Ellenberg_2_nil_quot_pi_better}) show that $$\delta_2(p + q) - \delta_2(p) - \delta_2(q) = [-,-]_* p \cup q,$$ where $[-,-]_*$ is the map on $\rH^2$ induced by the commutator $$[-,-]: \pi_1^{ab} \otimes \pi_1^{ab} \rightarrow  [\pi_1]_2/[\pi_1]_3$$ defined $$[\overline{\gamma},\overline{\ell}] \mapsto \gamma \ell \gamma^{-1} \ell^{-1},$$ where $\overline{\gamma} \in \pi_1^{ab}$ is the image of $\gamma \in \pi_1/[\pi_1]_3$ and similarly for $\ell$. It follows that $\delta_2$ is the sum of a cup product term and a linear term, after inverting $2$. For $X=\pmk$ based at $\01$, the linear term vanishes and we can avoid inverting $2$ by slightly changing what is meant by the cup-product term. This was shown by Ellenberg, who gave a complete calculation of $\delta_2$ in this case \cite[p. 11]{Ellenberg_2_nil_quot_pi_better}. Here is an alternative calculation of this $\delta_2$, showing the same result: let $\pi=\pi_1^{et}(\pmkbar, \01) \cong \langle x , y \rangle^{\wedge}$. Identify $\pi/[\pi]_2$ with $\Zhat (1) \oplus \Zhat (1)$ using the basis $\{x,y\}$, and identify $[\pi]_2/[\pi]_3$ with $\Zhat (2)$ using the basis $\{[x,y]\}$, so $\delta_2$ is identified with a map $$H^1(G_k, \Zhat (1) \oplus \Zhat (1)) \rightarrow H^2(G_k, \Zhat (2)).$$

\begin{proposition} \label{delta2cocycle} Let $p(g)= y^{a(g)} x^{b(g)}$ be a $1$-cocycle of $G_k$ with values in $\pi/[\pi]_2$, so $$b,a: G_k \rightarrow \Zhat(1)$$ are the cocycles produced by the isomorphism $\pi/[\pi]_2 \cong \Zhat(1)x \oplus \Zhat(1)y$. Then $$ \delta_2p = b \cup a. $$\end{proposition}

\begin{proof}
Sending $y^a x^b \in \pi/[\pi]_2$ to $y^a x^b \in \pi/[\pi]_3$ determines a set-theoretic section $s$ of the quotient map $\pi/[\pi]_3 \rightarrow \pi/[\pi]_2$. Then $\delta_2(p)$ is represented by the cocycle 
$$
(g,h) \mapsto \delta_2(p)(g,h) = s(p(g)) g s(p(h)) s(p(gh))^{-1}
$$ 
Using \eqref{G_k_action_F_2} and since $\frakf(g) \in [\pi]_2$ is mapped to a central element in $\pi/[\pi]_3$ we find
\[
\delta_2(p)(g,h) = \big(y^{a(g)}x^{b(g)}\big)\Big( \big(\frakf(g)^{-1}y^{\chi(g)} \frakf(g)\big)^{a(h)} x^{\chi(g)b(h)} \Big)
\big(y^{a(gh)} x^{b(gh)}\big)^{-1}
\]
\[
= \big(y^{a(g)}x^{b(g)}\big) \big(y^{\chi(g)a(h)} x^{\chi(g)b(h)}\big)\big(x^{-b(g) -\chi(g)b(h)} y^{-a(g) - \chi(g)a(h)}\big)
\]
\[
= y^{a(g)} [x^{b(g)}, y^{\chi(g)a(h)}]  y^{-a(g)} =  [x^{b(g)}, y^{\chi(g)a(h)}]  = [x,y]^{b(g) \cdot \chi(g)a(h)} 
= [x,y]^{(b \cup a)(g,h)}
\]
giving the desired result.
\qed \end{proof}

Proposition \ref{delta2cocycle} characterizes the lifts to $\rH^1(G_k, \pi/[\pi]_3)$ of an element of $$\rH^1(G_k, \pi^{ab}) \cong \rH^1(G_k, \Zhat(1)) \oplus \rH^1(G_k, \Zhat(1)):$$ let $b,a: G_k \rightarrow \Zhat(1)$ be cochains. For any $c \in C^1(G_k, \Zhat(2))$, define 
\begin{equation}\label{def(b,a)_c}
(b,a)_c: G_k \rightarrow \pi/[\pi]_3 \quad \textrm{ by } \quad (b,a)_c(g) = y^{a(g)}x^{b(g)}[x,y]^{c(g)}
\end{equation}

\begin{corollary}\label{lifts(b,a)_c_of(b,a)}Let $p(g)= y^{a(g)} x^{b(g)}$ be a $1$-cocycle of $G_k$ with values in $\pi/[\pi]_2$. The lifts of $p$ to a cocycle in $C^1(G_k, \pi/[\pi]_3)$ are in bijection with the set of cochains $c \in C^1(G_k, \Zhat(2))$ such that $$Dc = -b \cup a$$ by $$c \leftrightarrow (b,a)_c $$\end{corollary}

\begin{proof}
$D((b,a)_c) = \delta_2(p) + Dc$, where $D((b,a)_c)$ is as above (cf. \ref{defC*(G,A)}), and $\delta_2(p)$ denotes its cocycle representative given in the proof of Proposition \ref{delta2cocycle}.
\qed \end{proof}

\subsection{The obstruction $\delta_3$ as a Massey product} 
Note that $\chi(g)-1$ is divisible by $2$ in $\Zhat$ for any $g \in G_k$, where $\chi$ denotes the cyclotomic character, allowing us to define $$\frac{\chi -1}{2}: G_k \rightarrow \Zhat(\chi) \quad \textrm{ by } \quad g \mapsto \frac{\chi(g) -1}{2} \in \Zhat(\chi).$$ For any compatible system of primitive $n^{th}$ roots of unity in $\Zhat(1)$ giving an identification $\Zhat(1) = \Zhat(\chi)$, and in particular for $(\zeta_n)$ determined by Remark \ref{Zhat1=Zhatchi}, we have \begin{equation}  \label{chi-1/2_is_-1_kinC} \{-1\} = \frac{\chi -1}{2} \end{equation}  in $H^1(G_k, \Zhat(1))$, where $\{-1\}$ denotes the image of $-1$ under the Kummer map. The equality \eqref{chi-1/2_is_-1_kinC} holds as an equality of cocycles in $C^1(G_k, \Zhat(1))$ when $\{-1\}$ is considered as the cocycle $G_k \rightarrow \Zhat(1)$ given by choosing as the $n^{th}$ root of $-1$, the chosen primitive $(2n)^{th}$ root of unity; this is shown by the calculation \[
\{-1\}(g) = \big(g(\zeta_{2n})/\zeta_{2n}\big)_n = \big(\zeta_{2n}^{\chi(g)-1}\big)_n = \big(\zeta_n^{\frac{\chi(g)-1}{2}}\big)_n,
\] where for an element $a \in \Zhat(1)$, the reduction of $a$ in $\Z/n(1)$ is denoted $(a)_n$.

\begin{definition} \label{profinite_bin_coef} 
\textbf{Profinite binomial coefficients} are the  maps ${ \choose m} : \Zhat \to \Zhat$ for $m \geq 0$ defined for $a  \in \Zhat$ with $a \equiv a_n \mod n$ by 
\[
{a  \choose m} \equiv  a_{m! n}(a_{m!n} -1)(a_{m!n} -2)\ldots (a_{m!n} - m + 1)/m!  \mod n
\]
for every $n \in \bN$.
\end{definition}

\begin{example}\label{Dbchoose2}
For a cocycle $b\in C^1(G_k, \Zhat(\chi))$, let ${b \choose 2}$ in $C^1(G_k, \Zhat(\chi^2))$ denote the cochain  given by $$g \mapsto {b(g) \choose 2}.$$ We have
$$
D{b \choose 2} = - (b + \frac{\chi -1}{2}) \cup b,$$ as shown the the computation:
\[
D{b \choose 2} (g,h)= {b(g) \choose 2} + \chi(g)^2 {b(h) \choose 2} -{b(gh) \choose 2} =
\]
\[
 \frac{b(g)(b(g) -1 )}{2} +  \frac{\chi(g)^2 b(h)(b(h) -1 )}{2} 
- \frac{(b(g) + \chi(g)b(h))(b(g) + \chi(g)b(h) -1)}{2} 
\]
\[
= - b(g)\chi(g)b(h) - \frac{\chi(g)^2 b(h) - \chi(g)b(h)}{2}
\]
\[
= -(b \cup b)(g,h) - (\frac{\chi - 1}{2}\cup b) (g,h).
\]
\end{example} 

\begin{example} \label{choose_rest_sqrt}\hidden{claim2_7/22/08b} Let $b$ be an element of $k^\ast$ with compatibly chosen $n^{th}$ roots $\sqrt[n]{b}$ in $\overline{k}$, giving a cocycle $b: G_k \rightarrow \Zhat(1)$ via the Kummer map. Identify $\Zhat(1) = \Zhat(\chi)$ with $(\zeta_n)_{n \in \Z_{>0}}$ from Remark \ref{Zhat1=Zhatchi}. When restricted to an element of $C^1(G_{k(\sqrt{b})}, \Z/2)$, 
\[ {b \choose 2} = \{ \sqrt{b} \} \quad \in C^1(G_{k(\sqrt{b})}, \Z/2), \] where $\{ \sqrt{b} \}$ denotes the image of $\sqrt{b}$ under the Kummer map, which is independent of the choice of $\sqrt{\sqrt{b}}$. To see this, note that $\big( \{b\}(g) \big)_4 = 2  \big( \{\sqrt{b}\}(g) \big)_4 $ is even, whence $\big( \{b\}(g) - 1\big)_2 =1,$ and the value of $\frac{1}{2}(\big( \{b\}(g) \big)_4)$ in $\Z/2$ is $\big( \{\sqrt{b}\}(g) \big)_2$. Here, as above, the reduction mod $n$ of an element $a \in \Zhat$ is denoted $(a)_n$. \end{example}

\begin{remark} Note that ${b \choose 2}$ is a cochain taking values $\Zhat(\chi^2)$, but Example \ref{choose_rest_sqrt} identifies its image in $C^1(G_{k(\sqrt{b})}, \Z/2)$ with a cocycle taking values in $\Z/2(1)$. Furthermore, the cohomology class of the image of ${ b \choose 2}$ in $\rH^1(G_{k(\sqrt{b})}, \Z/2)$ depends on the choice of $\sqrt{b}$ used to define $b: G_k \rightarrow \Zhat(1)$. Nevertheless, ${ b \choose 2}$ appears in the expressions for $\delta_3$ which will be given in Proposition \ref{8/10_delta_3_cocycle_form}; it is involved in expressions which make the choice of $\sqrt[n]{b}$ irrelevant cf. Remark \ref{remarks_on_Prop_ploc_mod2_3nil_obs}\hidden{Also see Remark 5.17 3NilP1_1_28_11SV}. In writing down elements of $\pi$ in terms of $x$ and $y$, we have identified $\Zhat(1)$ and $\Zhat(\chi)$ because monodromy around $x$ distinguishes a compatible system of roots of unity.\hidden{One could change the root of unity by changing the embeddings $\C \supset \overline{\Q} \subseteq \kbar$, but this should just change the labelings of $x$ and $y$ in $\pi_1(\pmkbar)$.} \end{remark}

Identify $\Zhat(1)$ and $\Zhat(\chi)$ using $(\zeta_n)$ as in Remark \ref{Zhat1=Zhatchi}. In particular, we can apply profinite binomial coefficients to elements of $\Zhat(1)$ or any $\Zhat(n)$.

We define a $1$-cocycle $f(\sigma) \in C^1(G_k, \hat{\Z}(2))$ by 
 \begin{equation} \label{fdef}
\mathfrak{f}(\sigma) = [x,y]^{f(\sigma)} \mod [\pi]_3
\end{equation}
where $\mathfrak{f}(\sigma)$ is the $1$-cocycle from \eqref{G_k_action_F_2}  that describes the Galois action on $\pi$.

The basis $\{ [[x,y],x], [[x,y],y] \}$  for $[\pi]_3/[\pi]_{4}$ as a $\hat{\Z}$ module decomposes $\delta_3$ into two obstructions $$\delta_{3,[[x,y],x]},\delta_{3,[[x,y],y]}: \rH^1(G_k, \pi/[\pi]_3) \rightarrow \rH^2(G_k, \hat{\Z}(3)).$$ Since an arbitrary element of $\pi/[\pi]_3$ can be written uniquely in the form $y^a x^b[x,y]^c$ for $a,b,c \in \hat{\Z}$, an arbitrary element of $C^1(G_k, \pi/[\pi]_3)$ is of the form $(b,a)_c$, as in \eqref{def(b,a)_c}. The obstruction $\delta_3$ is therefore computed by the following: 

\begin{proposition} \label{delta_3_p1minus3_cocycle} Let $(b,a)_c$ be a $1$-cocycle for $G_k$ with values in $\pi/[\pi]_3$. Then:
\begin{description}
\item[(1)] $\delta_{3,[[x,y],x]} (b,a)_c$ is represented by the cocycle 
\begin{align*}
(g,h) \quad  \mapsto \quad & c(g)\chi(g) b(h) +  {b(g) + 1\choose 2} \chi(g) a(h)\\
 & + b(g) \chi(g)^2 a(h) b(h) -   \frac{\chi(g)-1}{2} \chi(g)^2 c(h),
 \end{align*}
 \item[(2)]
$ \delta_{3,[[x,y],y]} (b,a)_c$ is represented by the cocycle
\begin{align*}
(g,h) \quad \mapsto \quad & c(g)\chi(g) a(h) + b(g) {\chi(g)a(h) + 1\choose 2} \\
& -  \frac{\chi(g)-1}{2} \chi(g)^2c(h) - f(g) \chi(g)a(h).
 \end{align*}
 \end{description}
 \end{proposition} 
 
 \begin{proof}
We have the following equalities in $\pi/[\pi]_4:$
\begin{equation}
\label{switch_x_y_mod_pi_4}
x^b y^a = y^a x^b [x,y]^{ab}[[x,y],y]^{b {a+1 \choose 2}}[[x,y],x]^{a {b+1 \choose 2}}
\end{equation}
Replacing $b$ and $a$ by $-a$ in (\ref{switch_x_y_mod_pi_4}), yields: 
\begin{equation}
\label{[xg,yg]to[x,y]something}
[x^a,y^a] = [x,y]^{a^2}[[x,y],x]^{-a{a \choose 2}}[[x,y],y]^{-a{a \choose 2}}
\end{equation}

For any $g \in G_k$, a straightforward computation using (\ref{G_k_action_F_2}), (\ref{[xg,yg]to[x,y]something}) and (\ref{fdef}) shows that: 
\begin{align}
\label{g(elt)}
 g(y^a x^b[x,y]^c) =  \\
 \notag y^{\chi(g)a}x^{\chi(g)b}[x,y]^{\chi(g)^2 c}[[x,y],x]^{-\frac{\chi(g)-1}{2}\chi(g)^2 c}&[[x,y],y]^{-\frac{\chi(g)-1}{2}\chi(g)^2 c-f(g)\chi(g)a}
\end{align} 

An arbitrary element of $\pi/[\pi]_3$ can be written uniquely in the form $y^a x^b[x,y]^c$ for $a,b,c \in \hat{\Z}$. Sending $y^a x^b[x,y]^c \in \pi/[\pi]_3$ to $y^a x^b[x,y]^c \in \pi/[\pi]_4$ determines a section $s$ of the quotient map $\pi/[\pi]_4 \rightarrow \pi/[\pi]_3$. Let $p=(b,a)_c$. $\delta_3 p$ is represented by the cocycle $$(g,h) \mapsto s(p(g)) g s(p(h)) s(p(gh))^{-1}$$ which gives cocycles representing $\delta_{3,[[x,y],x]} p$ and $\delta_{3,[[x,y],y]} p$. Combining (\ref{switch_x_y_mod_pi_4}) and (\ref{g(elt)}) gives the desired result.
\qed \end{proof}

We give a formula for $\delta_3$ in terms of triple Massey products of elements of $\rH^1(G_k, \Zhat(1))$.

\begin{definition} \label{triple_Massey_prod_cocycle_def} The triple Massey product $\langle \alpha, \beta, \gamma \rangle$ for $1$ cocycles $\alpha,\beta,\gamma$ such that $\alpha \cup \beta =0$ and $\beta \cup \gamma = 0$ is described by choosing cochains $A,B$ such that $DA = \alpha \cup \beta$ and $DB = \beta \cup \gamma,$ and setting $\langle \alpha, \beta, \gamma \rangle = A \cup \gamma + \alpha \cup B.$  The choice $\{A,B\}$ is the called the defining system. The triple Massey product determines a partially defined multivalued product on $\rH^1$. \end{definition} 

\begin{remark}
Results of Dwyer and Stallings \cite{Dwyer} relate the element of $$\rH^2( \pi/[\pi]_{n}, [\pi]_n/[\pi]_{n+1})$$ classifying the central extension $$1 \rightarrow [\pi]_n/[\pi]_{n+1} \rightarrow \pi/[\pi]_{n+1} \rightarrow \pi/[\pi]_{n} \rightarrow 1$$ to $n^{th}$ order Massey products. The computation of $\delta_3$ is equivalent to computing the element of $\rH^2( \pi/[\pi]_{3} \rtimes G_k, [\pi]_3/[\pi]_{4})$ classifying $$1 \rightarrow [\pi]_3/[\pi]_{4} \rightarrow \pi/[\pi]_{4} \rtimes G_k \rightarrow \pi/[\pi]_{3} \rtimes G_k\rightarrow 1.$$ \end{remark}

Because $Dc = -b \cup a$ and $D{b + 1 \choose 2} = - (b - \frac{\chi -1}{2}) \cup b$ (by the same argument in Example \ref{Dbchoose2}), the expression for $\delta_{3,[[x,y],x]} (b,a)_c$ given in Proposition \ref{delta_3_p1minus3_cocycle} looks similar to a triple Massey product $\pm \langle \{\pm b\},\{ \pm b\}, a \rangle$, except the $c \cup b$ term should be $b \cup c$. The cup product on cohomology is graded commutative, and the analogue on the level of cochains, given below in Lemma \ref{D(cb)}, allows us to change the order of $c$ and $b$, which will express $\delta_{3,[[x,y],x]} (b,a)_c$ as a Massey product. 

For cochains $c \in C^1(G_k, \Zhat(n))$ and $b \in  C^1(G_k, \Zhat(m))$, define $$cb: G_k \rightarrow \Zhat(n+m) \quad \textrm {by } (cb)(g) = c(g)b(g).$$

 \begin{lemma}\label{D(cb)} Let $c \in C^1(G_k, \Zhat(n))$ be an arbitrary cochain, and $b$ in $C^1(G_k, \Zhat(m))$ be a cocycle. Then $$(D(cb) + b \cup c + c \cup b)(g,h) = Dc(g,h) b(g) + Dc(g,h)\chi(g)^{m} b(h)$$\end{lemma}

 \begin{proof}
By definition of $D$, for any $g,h \in G_k$, $$D(cb)(g,h) = (cb)(g) + \chi(g)^{n+m}(cb)(h) - (cb)(gh),$$ $$ c(gh) = c(g) + \chi(g)^n c(h) - Dc(g,h).$$ Since $b$ is a cocycle, $b(gh) = b(g) + \chi(g)^m b(h).$ Combining equations, we have \begin{align*} D(cb)(g,h) =& c(g)b(g) + \chi(g)^{n+m}c(h)b(h) \\ & - (c(g) + \chi(g)^n c(h) - Dc(g,h))(b(g) + \chi(g)^m b(h)) \\ 
 =& Dc(g,h) b(g) + Dc(g,h)\chi(g)^m b(h) - (c \cup b (g,h)+ b \cup c (g,h)) \end{align*}
 \qed \end{proof}

\begin{proposition} \label{8/10_delta_3_cocycle_form}
Let $p =(b,a)_c \in C^1(G_k, \pi/[\pi]_3)$ be a $1$-cocycle, where $(b,a)_c$ is as in the notation of \eqref{def(b,a)_c}. 
Then the following holds.
\begin{description}
\item[(1)] $\delta_{3,[[x,y],x]} (p) = - (b + \frac{\chi -1}{2}) \cup c - {b \choose 2} \cup a$,
 \item[(2)] $\delta_{3,[[x,y],y]} (p) =(a + \frac{\chi -1}{2}) \cup (ab-c) + {a \choose 2} \cup b - f \cup a$.
 \end{description}
 \end{proposition}
  
 \begin{proof}
 By Corollary \ref{lifts(b,a)_c_of(b,a)} and Lemma \ref{D(cb)}, we have that $-D(cb) (g,h) =$ \[ c(g) \chi(g) b(h) + b(g) \chi(g)^2 c(h) + b(g)^2 \chi(g) a(h) + b(g) \chi(g)^2 a(h)b(h).\] Subtracting this expression for $-D(cb)$ from the cocycle representing $\delta_{3,[[x,y],x]}p$ given in Proposition \ref{delta_3_p1minus3_cocycle} shows that $\delta_{3,[[x,y],x]} p$ is represented by the cocycle sending $(g,h)$ to
  \begin{align*}
 -b(g) \chi(g)^2 c(h) - b(g)^2\chi(g)a(h) +  {b(g) + 1\choose 2} \chi(g) a(h)
  -  \frac{\chi(g) -1}{2} \chi(g)^2 c(h). 
 \end{align*}
 Note that \[ - b(g)^2\chi(g)a(h) +  {b(g) + 1\choose 2} \chi(g) a(h) = \frac{ b(g)(-b(g)+1)}{2}\chi(g)a(h)\]\[ = - \big({b \choose 2} \cup a\big) (g,h ).\] Therefore $$\delta_{3,[[x,y],x]} p = - b \cup c - {b \choose 2} \cup a - \frac{\chi-1}{2} \cup c, $$ giving the claimed expression for $\delta_{3,[[x,y],x]} p$.
 
The claimed expression for $\delta_{3,[[x,y],y]}p$ follows from this formula for $\delta_{3,[[x,y],x]}p$ and a symmetry argument. Consider the action of $G_k$ on the profinite completion of the free group on two generators $F_2^{\wedge} = \langle x,y \rangle^{\wedge}$ given by $$g(x) = x^{\chi(g)} $$ $$g(y) = y^{\chi(g)}.$$ The $G_k$ action on $\pi$ described by (\ref{G_k_action_F_2}) would reduce to this action on $F_2^{\wedge}$ if $\mathfrak{f}$ where in the center of $\pi$. In particular, sending $x$ to $x$ and $y$ to $y$ induces isomorphisms of profinite groups with $G_k$ actions $$\pi/[\pi]_3 \cong  F_2^{\wedge}/[F_2^{\wedge}]_3$$ $$ [\pi]_3/[\pi]_4 \cong [F_2^{\wedge}]_3 / [F_2^{\wedge}]_4$$ Furthermore, viewing $f$ as a formal variable in the proof of Proposition \ref{delta_3_p1minus3_cocycle}, we see that Proposition \ref{delta_3_p1minus3_cocycle} implies that these isomorphisms fit into the commutative diagram \begin{equation}\label{delta_3pi-F_2}\xymatrix{ \rH^1(G_k, \pi/[\pi]_3)\ar[d]^{\cong} \ar[r]^{\delta_3 + \frak{f} \cup a }& \rH^2(G_k, [\pi]_3/[\pi]_4)\\
\rH^1(G_k, F_2^{\wedge}/[F_2^{\wedge}]_3) \ar[r]^{\delta_3 }& \rH^2(G_k, [F_2^{\wedge}]_3/[F_2^{\wedge}]_4) \ar[u]^{\cong}} \end{equation}
Let $i: F_2^{\wedge} \rightarrow F_2^{\wedge}$ be the $G_k$ equivariant involution defined by $$i(x) = y$$ $$i(y) = x.$$ Note that $i$ induces an endomorphism of the short exact sequence of $G_k$ modules $$ 1 \rightarrow [F_2^{\wedge}]_3/[F_2^{\wedge}]_4 \rightarrow F_2^{\wedge}/ [F_2^{\wedge}]_4 \rightarrow F_2^{\wedge}/ [F_2^{\wedge}]_3 \rightarrow 1.$$ Thus we have a commutative diagram $$\xymatrix{ \rH^1(G_k, F_2^{\wedge}/ [F_2^{\wedge}]_3) \ar[r]^{\delta_3}  & \rH^2(G_k, [F_2^{\wedge}]_3/ [F_2^{\wedge}]_4)\\
\rH^1(G_k, F_2^{\wedge}/ [F_2^{\wedge}]_3) \ar[r]^{\delta_3} \ar[u]^{i_*}&\rH^2(G_k, [F_2^{\wedge}]_3/ [F_2^{\wedge}]_4) \ar[u]^{i_*}} $$
Since $i$ is an involution, so is $i_*$, whence $$\delta_3 = i_* \delta_3 i_* .$$ 
With respect to the decomposition
\begin{equation}\label{F2_34_oplus}
\rH^2(G_k, [F_2^{\wedge}]_3/ [F_2^{\wedge}]_4)  = \rH^2(G_k, \Zhat(3)) \cdot [[x,y],x] \oplus \rH^2(G_k, \Zhat(3)) \cdot [[x,y],y]. 
\end{equation} 
$i_*$ acts by the matrix 
$$ 
\left( \begin{array}{cc}
0 & -1 \\
-1 & 0  \end{array} \right),
$$
or in other terms, 
\begin{equation}\label{iondeltas}
\delta_{3,[[x,y],y]} = - \delta_{3,[[x,y],x]} i_*.
\end{equation}
The map $i_\ast$ on $\rH^1(G_k,F_2^{\wedge}/[F_2^{\wedge}]_3)$ we compute as
\[
i_\ast((b,a)_c)(g) = x^{a(g)}y^{b(g)}[y,x]^{c(g)}
\]
\[
= [x^{a(g)}, y^{b(g)}] y^{b(g)} x^{a(g)} [x,y]^{-c(g)} = y^{b(g)} x^{a(g)} [x,y]^{a(g)b(g)-c(g)} = (a,b)_{ab-c}(g).
\]
Combining \eqref{delta_3pi-F_2} with \eqref{iondeltas}
\[
\delta_{3,[[x,y],y]}((b,a)_c) = \delta_{3,[[x,y],y]}^{F_2^{\wedge}}((b,a)_c) - f \cup a =  - \delta_{3,[[x,y],x]}^{F_2^{\wedge}}(i_\ast(b,a)_c) - f \cup a 
\]
\[
= - \delta_{3,[[x,y],x]}^{F_2^{\wedge}}((a,b)_{ab-c}) - f \cup a = - \delta_{3,[[x,y],x]}((a,b)_{ab-c}) - f \cup a
\]
\[
=(a + \frac{\chi -1}{2}) \cup (ab-c)  +  {a \choose 2} \cup b  - f \cup a
\]
as claimed by the proposition. In the above manipulation, we have marked obstructions corresponding to $F_2^{\wedge}$ with a superscript to avoid confusion.

\qed \end{proof}
 
\begin{remark}\label{symmetry_vs_cocycle_delta_xyy} The above symmetry argument combined with the explicit cocycle for $\delta_{3,[[x,y],y]}$ given in Proposition \ref{delta_3_p1minus3_cocycle} gives unexploited computational information.\end{remark}

\begin{theorem}\label{delta_3_Massey_prod_8_10} Let $p$ be an element of $\rH^1(G_k, \pi/[\pi]_3)$, so $p$ is represented by a cocycle $(b,a)_c \in C^1(G_k, \pi/[\pi]_3)$ in the notation of \eqref{def(b,a)_c}. Then
$$ \delta_{3,[[x,y],x]} (p) = \langle (b+ \frac{\chi -1}{2}),b,a \rangle  \quad \quad \textrm{defining system:} \{-{b \choose 2}, -c \} $$
$$ \delta_{3,[[x,y],y]} (p) =- \langle (a+ \frac{\chi -1}{2}),a,b\rangle  - f \cup a  \quad \quad \textrm{defining system:} \{ -{a \choose 2}, c-ab\}.$$
In particular, for $(b,a) \in \Jac (\pmk)(k) = k^\ast \times k^\ast$, we have $$ \delta_{3,[[x,y],x]} (b,a) =  \langle \{-b\}, b, a \rangle$$ $$\delta_{3,[[x,y],y]} (b,a) = - \langle \{-a\}, a, b \rangle - f \cup a.$$ 
\end{theorem}

\begin{remark}\label{Massey_Thm_Remark} (i) As above, an element of $k^*$ also denotes its image in $\rH^1(G_k, \Zhat(1))$ under the Kummer map in the last two equations. The brackets in the notation $\{-b\},\{-a\}$ serve to distinguish between the additive inverse of $b$ in $H^1(G_k, \Zhat(1))$ and the image of $-b$ under the Kummer map. We note the obvious remark that the Kummer map is a homomorphism, because this will appear in (iii) of this remark on the level of cocycles; namely given $a,b \in k^\ast$ with compatibly chosen $n^{\rm th}$ roots $\sqrt[n]{a}$ and $\sqrt[n]{b}$, then $\{a\}+ \{b\}: G_k \rightarrow \Zhat(1)$ is the image under the Kummer map of $ab$ with $\sqrt[n]{a}\sqrt[n]{b}$ chosen as the $n^{th}$ root of $ab$. This means that if one has chosen compatible primitive roots of $-1$, as is the case by \eqref{chi-1/2_is_-1_kinC} and Remark \ref{Zhat1=Zhatchi}, the cocycle $\{ - 1 \} + \{a\}$ is different from $- \{-1\} + \{a \} $ although both give the same class in cohomology, namely the class $\{ -a\}$. 


(ii) Expressing $\delta_3$ in terms of Massey products reduces the dependency on $c$ to the choices of the defining systems. In fact, after restricting to defining systems of the appropriate form, the choice of these defining systems and the choice of lift are equivalent. More explicitly, we will say that a choice for the defining systems $\{A,B\}$ of $\langle \{-x\}, x, y \rangle$ and $\{C,D\}$ of $\langle \{-y\}, y, x \rangle$ is \textit{compatible} if $B + D = - xy$, $A = -{x \choose 2}$, and $C = -{y \choose 2}$. The choice of defining system also encompasses choosing cocycle representatives for the cohomology classes involved; the cocycle representative for $\{-x\}$ is the representative for $x$ plus $\frac{\chi -1}{2}$, and similarly for $\{ -y\}$, as in (i). Then choosing a lift of $(b,a)$ in $\rH^1(G_k, \pi/[\pi]_2)$ to $(b,a)_c$ in $\rH^1(G_k, \pi/[\pi]_3)$ is equivalent to choosing compatible defining systems for $\langle \{-b\}, b, a \rangle$ and $\langle \{-a\}, a, b \rangle$, by Corollary \ref{lifts(b,a)_c_of(b,a)} and Theorem \ref{delta_3_Massey_prod_8_10}. As we are ultimately interested in obstructing points of the Jacobian from lying on the curve, it is natural to suppress both the defining system in the Massey product and the choice of lift, and view $\delta_3$ and triple Massey products as multivalued functions on $\rH^1(G_k, \pi/[\pi]_2)$.\end{remark}   

\begin{proof} Comparing the expression for $ \delta_{3,[[x,y],x]} (b,a)_c$ of Proposition \ref{8/10_delta_3_cocycle_form}, and the equations $D (-{b \choose 2}) =  (b + \frac{\chi -1}{2}) \cup b$ of Example \ref{Dbchoose2}, and $D(-c) = b\cup a$ of Corollary \ref{lifts(b,a)_c_of(b,a)} with the definition of the triple Massey product (Definition \ref{triple_Massey_prod_cocycle_def}) shows $$ \delta_{3,[[x,y],x]} (b,a)_c =  \langle (b + \frac{\chi -1}{2}), b, a \rangle$$ with the defining system $\{-{b \choose 2}, -c \} $.

By Lemma \ref{D(cb)}, $$- D(ab) = a \cup b + b \cup a ,$$ whence $D (c-ab) = a \cup b$ by Corollary \ref{lifts(b,a)_c_of(b,a)}. Comparing $D (-{a \choose 2}) =  (a + \frac{\chi -1}{2}) \cup a$, $D (c-ab) = a \cup b$, and the expression for $\delta_{3,[[x,y],y]}(b,a)_c$ of Proposition \ref{8/10_delta_3_cocycle_form} with the definition of the triple Massey product shows $$\delta_{3,[[x,y],y]} (b,a)_c = - \langle (a + \frac{\chi -1}{2}), a, b \rangle - f \cup a$$ with defining system $\{ -{a \choose 2}, c-ab \}$. By \eqref{chi-1/2_is_-1_kinC}, we have $a + \frac{\chi -1}{2} = \{ -a \}$ and $b + \frac{\chi -1}{2} = \{ -b \}$, showing the theorem. 
\qed \end{proof} 

\section{Evaluating $\delta_2$ on $\Jac(k)$}\label{sectiondelta_2} Let $k$ be a subfield of $\C$ or a completion of a number field equipped with $\C \supset \overline{\Q} \subseteq \kbar$ as in \ref{secpi1pmksec}. Let $X=\proj_k^1 - \{0,1,\infty\}$ and recall that in \ref{AJpmk_to_GmGm} we fixed an isomorphism $\Jac(X) = \G_{m,k} \times \G_{m,k}$. Let $\pi= \pi_1(X_{\kbar}, \01)$. In \eqref{pixywedeiso}, we specified an isomorphism $\pi=\langle x,y\rangle^{\wedge}$. 

The obstruction $\delta_2$ is given on $(b,a)$ in $\Jac(X)(k)$ by $\delta_2(b,a) = b \cup a$, so evaluating $\delta_2$ is equivalent to evaluating the cup product \begin{equation} \label{H1H2cup}\rH^1(G_k, \Zhat(1)) \otimes \rH^1(G_k, \Zhat(1)) \rightarrow \rH^2(G_k, \Zhat(2)).\end{equation}  Evaluating the obstruction $\delta_2^2$ coming from the lower exponent $2$ central series (cf. \ref{Ellenbergob_subsection}) is equivalent to evaluating the mod $2$ cup product \begin{equation}\label{mod2H1H2cup} \rH^1(G_k, \Z/2) \otimes \rH^1(G_k, \Z/2) \rightarrow \rH^2(G_k, \Z/2),\end{equation} and this evaluation is recalled for $k = \Q_p$, $\R$, and $\Q$ in \ref{cup_prod_table_Q_p}-- \ref{cup_prod_GQ_Z/2}. The remainder of Section \ref{sectiondelta_2} gives evaluation results for the obstruction $\delta_2$ itself.  From the bilinearity of $\delta_2$, Proposition \ref{some_pts_in_Kerdelta_2} finds infinite families of points of $\Jac (k)$ which are unobstructed by $\delta_2$, but which are not the image of a rational point or tangential point under the Abel-Jacobi map. This is rephrased as ``the $2$-nilpotent section conjecture is false" in \ref{2nilseccon} and Proposition \ref{false2nilscprop}. These families provide a certain supply of points on which to evaluate $\delta_3$.  The subsection \ref{delta2Qalg} contains a finite algorithm for determining if $\delta_2$ is zero or not for $k=\Q$, using Tate's calculation of $\rK_2(\Q)$, and gives Jordan Ellenberg's geometric proof that the cup product factors through $\rK_2$.

\subsection{The mod $2$ cup product for $k_v$} \label{cup_prod_table_Q_p} Let $p$ be an odd prime. Let $k_v$ be a finite extension of $\Q_p$ with valuation $v: k_v^\ast \surj \Z$, integer ring $\mathcal{O}_v$ and residue field $\mathbb{F}_v$. Let $\mathfrak{p}$ be a uniformizer of $\mathcal{O}_v$ and $u \in \mathcal{O}_v$ be a unit and not a square, so $\{ \mathfrak{p}, u \}$ is a basis for the $\Z/2$ vector space $$ k_v^\ast/(k_v^\ast)^2  \cong \rH^1(G_{k_v}, \Z/2),$$ where the isomorphism follows from the Kummer exact sequence \eqref{KummerSESmodn} and Hilbert 90. By Hilbert 90, we have that $\rH^2(G_{k_v}, \Z/2)$ is the $2$ torsion of the Brauer group $\rH^2(G_{k_v}, \overline{k_v}^\ast)$ of $k_v$, which is isomorphic to $\Q/\Z$ by the invariant (see \cite[VI]{CF}, for instance), so $$\rH^2(G_K, \Z/2) \cong (\Q/\Z) [2] = (\frac{1}{2}\Z)/\Z.$$ The (distributive and commutative) mod $2$ cup product \eqref{mod2H1H2cup} is given by the table
 \begin{center}
  \begin{tabular}{c|ccc}
    $\cup$                    & $u$ && $\mathfrak{p}$  \\\hline
    $u$                          & $0$    && $1/2$    \\
    $\mathfrak{p}$      & $1/2$  && $\{-1\} \cup \mathfrak{p}$    
  \end{tabular}
  \end{center} where $\{-1\} \cup \mathfrak{p} = 0$ if $-1$ is a square in $\mathbb{F}_v$ and $\{-1\} \cup \mathfrak{p} = 1/2$ otherwise. 
  
We include a derivation of this well-known calculation: as $\rH^2(G_{k_v}, \Z/2)$ injects into $\rH^2(G_{k_v}, \overline{k_v}^\ast)$, we may calculate in the Brauer group. For $(a,b) \in k_v^\ast \oplus k_v^\ast$, define $E_{(\sqrt{a},b)} \in C^1(G_{k_v}, \overline{k_v}^\ast)$ by $$E_{(\sqrt{a},b)}(\sigma) = (\sqrt{a})^{b(\sigma)},$$ where $b: G_{k_v} \rightarrow \{0,1\}$ is defined by $(-1)^{a(\sigma)} =  (\sigma \sqrt{a})/\sqrt{a}.$  A short calculation shows that $$D E_{(\sqrt{a},b)}(\sigma, \tau) = (-1)^{a(\sigma) b (\tau)} a^{b(\sigma)b(\tau)},$$ whence $a \cup b$ in $\rH^2(G_{k_v}, \overline{k_v}^\ast)$ is represented by \begin{equation}\label{acupb_H2(Gal(krtb/k))}(\sigma, \tau) \mapsto  a^{b(\sigma)b(\tau)}.\end{equation} Let $k_v^{nr}$ denote the maximal unramified extension of $k_v$, and $v: (k_v^{nr})^\ast \rightarrow \Z$ the extension of the valuation. For $b=u$, the cocycle \eqref{acupb_H2(Gal(krtb/k))} factors through $\Gal({k_v}^{nr}/k_v)^2$, and by \cite[Chap VI 1.1 Thm 2 pg 130]{CF}, $$v_*: \rH^2(\Gal({k_v}^{nr}/k_v), (k_v^{nr})^\ast)  \rightarrow \rH^2(\Gal(k_v^{nr}/k_v), \Z)$$ is an isomorphism, showing that $\{u \} \cup \{u\} = 0$, and the invariant of $ \{\mathfrak{p}\} \cup \{u \}$ is $1/2$ as claimed (see \cite[pg 130]{CF}). To compute $\mathfrak{p} \cup \mathfrak{p}$, note that for $a = -1$, the cochain $(\sigma, \tau) \mapsto  a^{b(\sigma)b(\tau)}$ equals $b \cup b$. By the above, $a \cup b$ is also represented by $(\sigma, \tau) \mapsto  a^{b(\sigma)b(\tau)}$, so it follows that $$\{-1\} \cup \mathfrak{p} = \mathfrak{p} \cup \mathfrak{p}.$$ 
  


It follows that $\delta_2$ is non-trivial: let $X= \pmQ$. Let $\delta_2^{(2,p)}$ denote $\delta_2^2$ for $k = \Q_p$ and $X_{\Q_p}$. Consider $\delta_2^{(2,p)}$ as a function on $\Jac(X)(\Q)$ by evaluating $\delta_2^{(2,p)}$ on the corresponding $\Q_p$ point of $\Jac(X_{\Q_p})$.

\begin{corollary}
Choose $x,y$ in $\Q^*$. Let $p$ be an odd prime and $u$ be an integer which is not a quadratic residue mod $p$. Then $\delta_2^{(2,p)}(uy^2,px^2) \neq 0$, and therefore $\delta_2^2(uy^2,px^2) \neq 0$ and $\delta_2(uy^2,px^2) \neq 0$.
\end{corollary}

\subsection{The mod $2$ cohomology of $G_{\R}$}\label{H2(G_R,Z/2)}\label{R_cup_prod} By \cite[III Example 3.40 p. 250]{Hatcher} or \cite[III.1 Ex 2 p.58 p.108]{Brown_coh_groups}, there is a ring isomorphism
$$
\rH^*(G_{\R}, \Z/2) \cong \rH^*(\Z/2, \Z/2) \cong \Z/2[\alpha]
$$ 
where $\alpha$ is the nontrivial class in degree $1$, namely 
\[
\alpha = \{-1\} \in \R^\ast/(\R^\ast)^2 = \rH^1(G_{\R},\mu_2) = \rH^1(G_\R,\Z/2\Z).
\]
In particular, the cup product is an isomorphism
$$
\rH^1(G_{\R},\Z/2\Z(1)) \otimes \rH^1(G_{\R},\Z/2\Z(1)) \rightarrow \rH^2(G_{\R},\Z/2\Z(2)). 
$$ 

The map $C^2(G_{\R},\Z/2) \rightarrow \Z/2$ given by evaluating a cochain at $(\tau,\tau)$ for $\tau$ the non-trivial element of $G_{\R}$ determines an isomorphism $\rH^2(G_{\R},\Z/2) \rightarrow \Z/2$.


Let $X= \pmQ$, and let $\delta_2^{(2,\R)}$ denote $\delta_2^{2}$ for $k = \R$ and $X_{\R}$. Consider $\delta_2^{(2,\R)}$ as a function on $\Jac(X)(\Q)$ by evaluating $\delta_2^{(2,\R)}$ on the corresponding $\R$ point of $\Jac(X_{\R})$.

\begin{corollary} Let $b,a$ be elements of $\Q^*$. Then $\delta_2^{(2,\R)}(b,a) \neq 0$ if and only if $a,b <0$. If $a,b<0$, then $\delta^2_2 (b,a) \neq 0$ and $\delta_2 (b,a) \neq 0$.\end{corollary}

\subsection{The mod $2$ cup product for $G_{\Q}$}\label{cup_prod_GQ_Z/2}
 There is a finite algorithm for computing the cup product of two Kummer classes  
$$
\Q^* \otimes \Q^* \rightarrow \rH^1(G_{\Q}, \Z/2\Z(1)) \otimes \rH^1(G_{\Q}, \Z/2\Z(1)) \rightarrow \rH^2(G_{\Q}, \Z/2\Z(2)).
$$ 
The phrase ``computing an element of $\rH^2(G_{\Q}, \Z/2\Z(2))$" means the following. By the local-global principle for the Brauer group, the theorem of Hasse, Brauer, and Noether, see \cite[8.1.17 Thm p. 436]{coh_num_fields}, we have an isomorphism 
\[
\rH^2(G_{\Q}, \Z/2\Z(2)) = \Br(\Q)_2 \otimes \mu_2 = \ker\big(\bigoplus_v \Br(\Q_v)_2 \xrightarrow{\sum_v \inv_v} \Q/\Z \big) \otimes \mu_2
\]
where $A_2$ is the $2$-torsion of an abelian group $A$ and $\inv_v: \Br(k_v) \inj \Q/\Z$ is the local invariant map for the Brauer group of the local field $k_v$, see \cite[VI]{CF}, and $v$ ranges over all places of $\Q$. The class in $\rH^2(G_{\Q}, \Z/2\Z(2))$ is therefore completely determined by its restrictions to local cohomology groups with the freedom to ignore one place by reciprocity. We will benefit from this freedom by ignoring the prime $2$ for which we did not describe the mod $2$ cup product computation above.

Given $b$ and $a$ in $\Q^*$, we give a finite algorithm for computing $\operatorname{inv}_{\nu} (\{b\} \cup \{a\})$ for every $\nu \neq 2$. By \eqref{cup_prod_table_Q_p}, for an odd prime $p$ with $p \nmid ab$ we have $\{b\} \cup \{a\} = 0$. It therefore remains to evaluate 
\[
\inv_v (\{b\} \cup \{a\})
\]
for $v = \R$ and the finitely many odd primes $v=p$ with $p \mid ab$.
This is accomplished in finitely many steps by \ref{cup_prod_table_Q_p} and \ref{R_cup_prod}.

In fact, given any field extension $\Q \subset E$ and any cocycle in $C^2(\Gal(E/\Q), \Z/2)$ (for instance the ones given in Proposition \ref{8/10_delta_3_cocycle_form}), there is a finite algorithm for computing the associated element of $H^2(G_{\Q}, \Z/2)$.

\subsection{The $2$-nilpotent section conjecture for number fields is false}\label{2nilseccon} We describe several families of points of $\Jac(\pmk)(k)$ such that $\delta_2$ vanishes. 

\begin{example} \label{ex:familiesofkerdelta2}
(1) The map $\delta_2$ vanishes on the $k$ points and tangential points of the curve $\pmk$ by design. Therefore, the $k$ points of $$\Jac(\pmk) = \G_{m,k} \times \G_{m,k}$$ of the form $(x, 1-x)$ or $(-x,x)$ satisfy $\delta_2 = 0$ by \eqref{aj(x)_is_(x,1-x)} and Lemma~\ref{AbelJacobi(tgtpnt)}. 

From a more computational point of view, the vanishing of $\delta_2$ on $(-x,x)$ follows from the the calculation in Lemma~\ref{Dbchoose2} identifying the cochain whose boundary is $\{-x\} \cup \{x\}$ (use Lemma~\ref{chi-1/2_is_-1_kinC} and~\ref{addingKummercocycles}, to identify $\{-x\}$ and $\{x\} + \frac{\chi-1}{2}$). I do not presently see a specific cochain in $C^1(G_k, \Zhat(2))$ whose boundary is $\{x\} \cup \{1-x\}$, although I would not be surprised if such a cochain could be written down explicitly. 

(2) Since $\delta_2(b,a) = \{b\} \cup \{a\}$ by Proposition~\ref{delta2cocycle}, the map $\delta_2$ is bilinear in both coordinates of  
$$
k^* \oplus k^* = (\G_m \times \G_m)(k) = \Jac(\pmk)(k)
$$ 
Therefore, for any $(b,a)$ in $k^* \oplus k^*$, such that $\delta_2 (b,a)=0$, the points of the form $(b^m,a^n)$, for integers $m$ and $n$, satisfy $\delta_2(b^m,a^n) = 0$ as well. Likewise for any $(b,a)$, $(b,c)$ such that $\delta_2 (b,a)= \delta_2 (b,c)=0$, the point of the Jacobian $(b,ac)$ also satisfies $\delta_2(b,ac) = 0$. As was pointed out by Jordan Ellenberg, this produces many families of $k$ points of $\Jac(\pmk)$ which are unobstructed by $\delta_2$. A special case of this is mentioned in Proposition~\ref{some_pts_in_Kerdelta_2} below. 
\end{example}

\begin{proposition} \label{some_pts_in_Kerdelta_2}
Let $X=\pmk$. For any $x$ in $k^*$, the map $\delta_2$ vanishes on 
\[
(x, (1-x)^m), \quad ((-x)^m,x), \quad ((1-x)(-x),x), 
\]
\[
(x^{n_1}(1-((1-x)^{n_2}(-x)^{n_3})), (1-x)^{n_2}(-x)^{n_3}).
\]
\end{proposition}

\begin{proof}
This follows immediately from the discussion in Example~\ref{ex:familiesofkerdelta2} above.
\qed \end{proof}

Say the ``$n$-nilpotent section conjecture" for a smooth curve $X$ over a field $k$ holds if the natural map from $k$ points and tangential points to conjugacy classes of sections of \begin{equation}\label{absectionsfirst} 1 \rightarrow \pi_1(X_{\kbar})^{\ab} \rightarrow \pi_1(X)/ [ \pi_1(X_{\kbar})]_2 \rightarrow G_k \rightarrow 1 \end{equation} which arise from $k$ points of $\Jac X$, and lift to sections of \begin{equation} \label{nnilsections} 1 \rightarrow \pi_1(X_{\kbar})/[ \pi_1(X_{\kbar})]_{n+1} \rightarrow  \pi_1(X)/[\pi(X_{\kbar})]_{n+1} \rightarrow G_k \rightarrow 1 \end{equation} is a bijection. 
This is similar to the notion of ``minimalistic" birational section conjectures introduced by Florian Pop \cite{Pop_min_sc}. The notion given here has the disadvantage that it mentions the points of the Jacobian, and is therefore not entirely group theoretic. The reason for this is that finite nilpotent groups decompose as a product of $p$ groups, so the sections of (\ref{absectionsfirst}) and (\ref{nnilsections}) decompose similarly, allowing for sections which at different primes arise from different rational points of $X$. Restricting to a single prime will not give a ``minimalistic" section conjecture by results of Hoshi \cite{Hoshi_non-geom_pro-pGalois_sec}. 

Because the conjugacy classes of sections of a split short exact sequence of profinite groups $$1 \rightarrow Q \rightarrow Q \rtimes G \rightarrow G \rightarrow 1 $$ are in natural bijective correspondence with the elements of $\rH^1(G,Q)$\hidden{naturality comes from having chosen a splitting}, the $n$-nilpotent section conjecture for a curve with a rational point is equivalent to: {\em the natural map from $k$ points and tangential points to the kernel of $\delta_n$ is a bijection,} where the kernel of $\delta_n$ is considered as a subset of $\Jac X$. More precisely, a smooth, pointed curve $X$ gives rise to a commutative diagram
\[
\xymatrix{
 \rH^1(G_k,\pi/[\pi]_n) \ar[r]^{\delta_n} \ar[d]^{pr} & \rH^2(G_k, [\pi]_n/[\pi]_{n+1}) \\
 \rH^1(G_k,\pi/[\pi]_2) && X(k) \cup \bigcup_{x \in \overline{X}-X}(T_x \overline{X}(k)-\{0\}) \ar[ll]^{\alpha}\\
  \Jac(X)(k) \ar[u]^{\kappa}
}
\] where $\overline{X}$ denotes the smooth compactification of $X$, and $\pi$ denotes the fundamental group of $X_{\kbar}$. The $n$-nilpotent section conjecture is the claim $\alpha$ induces a bijection
\[
X(k) \cup \bigcup_{x \in \overline{X}-X}(T_x \overline{X}(k)-\{0\}) \rightarrow \kappa(\Jac(X)(k)) \cap pr(\ker(\delta_n)).
\]

Say the ``$n$-nilpotent section conjecture" holds for a field $k$, if for all smooth, hyperbolic curves over $k$, the $n$-nilpotent section conjecture holds. 

\begin{proposition} \label{false2nilscprop}
The $2$-nilpotent section conjecture fails for any subfield $k$ of $\C$ or $k$ the completion of a number field. In fact, the $2$-nilpotent section conjecture does not hold for $\pmk$. 
\end{proposition}

\begin{proof}
Choose $x \not= 1$ in $k^*$ such that $(1-x)^2$ does not equal $-x$. Then the $k$ point of $\Jac (\pmk)$ given by $(x,(1-x)^2)$ is not the image of a $k$ point or tangential point of $\pmk$ by Lemmas~\ref{aj(x)_is_(x,1-x)} and~\ref{AbelJacobi(tgtpnt)}. By Proposition~\ref{some_pts_in_Kerdelta_2}, the section of (\ref{absectionsfirst}) determined by $(x,(1-x)^2)$ lifts to a section of (\ref{nnilsections}) for $n=2$.
\qed \end{proof}

While the failure of a  ``minimalistic section conjecture"  is not surprising at all, some do hold.
Moreover, the ``minimalistic section conjectures" in \cite{Pop_min_sc} \cite{delta2real} do not mention the points of the Jacobian and so control the rational points of $X$ group theoretically.

\subsection{The obstruction $\delta_2$ over $\Q$}\label{delta2Qalg}

By \cite{TateK2} Theorem 3.1, the cup product \eqref{H1H2cup} composed with the Kummer map~\eqref{eq:Kummermap} \begin{equation}\label{cup_prod_Kummer}
k^\ast \otimes k^\ast \rightarrow \rH^2(G_k,\Zhat(2))
\end{equation} factors through
the Milnor $K_2$-group of $k$ $$\rK_2(k) =  k^* \otimes_{\Z} k^* / \langle x \otimes (1-x): x \in k^\ast \rangle$$ mapping to $\rH^2(G_k,\Zhat(2))$ by the Galois symbol $h_k: \rK_2(k) \to \rH^2(G_k,\Zhat(2))$.

\begin{proposition} \label{Kerdelta_2_isKerK2}
Let $k$ be a finite extension of $\Q$, and $X = \pmk$. 
For any $(b,a) \in \Jac(X)(k) = k^\ast \times k^\ast$ we have $\delta_2(b,a) = 0$ if and only if $b \otimes a = 0 $ in $\rK_2(k)$.
\end{proposition}

\begin{proof}
Since $\delta_2(b,a) = b \cup a$ by Proposition~\ref{delta2cocycle}, we have that $\delta_2$ factors through $h_k$. By \cite{TateK2} Theorem 5.4, $h_k$ is an isomorphism onto the torsion subgroup of $ \rH^2(G_k,\Zhat(2))$. 
\qed
\end{proof}

Tate's computation of $\rK_2 (\Q)$ thus gives an algorithm for computing $\delta_2$ for $\pmQ$ on any rational point $(b,a)$ of $\Jac (\pmQ)(\Q)= \Q^\ast \times  \Q^\ast$: by \cite[Thm 11.6]{Milnor_Kthy}, there is an isomorphism 
$$ 
\rK_2(\Q) \cong \mu_2 \oplus_{p \textrm{ odd prime}} \bF_p^\ast 
$$ 
given by the tame symbols: for odd $p$ (with $p$-adic valuation $v_p$)
$$
\rK_2(\Q) \rightarrow \bF_p^\ast
$$ 
\[
x \otimes y \mapsto (x,y)_p = (-1)^{v_p(x) v_p(y)} {x^{v_p(y)}}{y^{-v_p(x)}}  \in  \bF_p^\ast,
\]
and a map at the prime $2$
$$
\rK_2(\Q) \rightarrow \mu_2
$$ 
\[
x \otimes y \mapsto (x,y)_2 = (-1)^{iI+jK+kJ} 
\]
where $x = (-1)^i 2^j 5^k u$ and $y = (-1)^I 2^J 5^K u'$ with $k,K  = 0$ or $1$, and $u,u'$ quotients of integers congruent to $1$ mod $8$. 

By Proposition~\ref{Kerdelta_2_isKerK2}, we have $\delta_2(b,a) = 0$ if and only if $(b,a)_p=0$ for $p=2$ and for $p$ equal to all odd primes dividing $a$ or $b$. 

\begin{example}\label{(p,-p)=0inK2Qnew} As an example of this algorithm, consider $\delta_2(p,-p)$ for $p$ an odd prime: for $q $ different from $p$ and $2$, we have $(p,-p)_q=1$ because $v_q(p)$ and $v_q(-p)$ vanish. At the prime $p$, we have $(p,-p)_p = (-1)^1 p/ (-p) = 1$. For the prime $2$, express $p$ in the form $p = (-1)^i2^j5^ku$ with $k = 0,1$ and $u$ a quotient of integers congruent to $1$ mod $8$. Then $-p =(-1)^{i+1}2^j5^ku$, and $$(p,-p)_2 = (-1)^{i(i+1) + 2 jk} = 1.$$ Thus $\delta_2(p,-p)=0$, as also followed from Lemma \ref{AbelJacobi(tgtpnt)}, as $(p,-p)$ is the image of a tangential point of $\pmQ$, or can be deduced from the Steinberg relation. \end{example}

\begin{remark}\label{delta2K2sec} The computation of $\delta_2$ for $\pmk$ gives a geometric proof that (\ref{cup_prod_Kummer}) factors through $\rK_2(k)$, as observed by Jordan Ellenberg. Namely, by construction $\delta_2$ vanishes on the image of $$\pmk(k) \rightarrow \Jac (\pmk) (k) = k^* \times k^*.$$ By \eqref{aj(x)_is_(x,1-x)}, this image is $(x, 1-x)$, and by Proposition~\ref{delta2cocycle}, the obstruction $\delta_2$ is given by $\delta_2(b,a) = b \cup a$. Thus, the map (\ref{cup_prod_Kummer}) vanishes on $x \otimes (1-x)$, and therefore factors through $\rK_2(k)$.\end{remark}

\section{Evaluating quotients of $\delta_3$}\label{delta3sec}

We keep the notation $k$, $X = \pmk$, and $\pi=\pi_1(\pmkbar, \01)$ from \ref{sectiondelta_2}. In particular, we have a chosen isomorphism $\pi= \langle x,y \rangle^{\wedge}$ as above. To evaluate $\delta_3$ on points in $ \Ker \delta_2$ requires Galois cohomology computations with coefficients in $[\pi]_3/[\pi]_4 \cong \Zhat(3) \oplus \Zhat(3)$. As computing in $\rH^2(G_k, \Zhat(3))$ seems difficult, we will evaluate a quotient of $\delta_3$ which can be computed in $\rH^2(G_k, \Z/2(3)) = \rH^2(G_k, \Z/2)$. This quotient is denoted $\delta_3^{\mdl 2}$ and can be described as ``the reduction of $\delta_3$ mod $2$" as well as ``the $3$-nilpotent piece of $\delta_3^2$," where $\delta_3^2$ is the obstruction coming from the lower exponent $2$ central series as in \ref{Ellenbergob_subsection}.

\subsection{Definition of $\delta_3^{\mdl 2}$}\label{sec_defndelta_3mod2} Composing the obstruction $$\delta_3: \rH^1(G_k, \pi/[\pi]_3) \rightarrow \rH^2(G_k,  [\pi]_3/[\pi]_4)$$ with the map on $\rH^2$ induced from the quotient $$\xymatrix{[\pi]_3/[\pi]_4 \cong \Zhat(3)[[x,y],x] \oplus \Zhat(3) [[x,y],y] \ar@{->>}[d] \\ [\pi]_3/[\pi]_4 ([\pi]_3)^2 \cong \Z/2\Z[[x,y],x] \oplus \Z/2\Z [[x,y],y]}$$ gives a map $\rH^1(G_k, \pi/[\pi]_3) \rightarrow \rH^2(G_k,  [\pi]_3/[\pi]_4 ([\pi]_3)^2)$ which factors through $$\rH^1(G_k, \pi/[\pi]_3) \rightarrow \rH^1(G_k, \pi/[\pi]_3^2)$$ (as follows from Proposition \ref{8/10_delta_3_cocycle_form}) where $[\pi]_n^2$ denotes the $n^{th}$ subgroup of the lower exponent $2$ central series (cf. \ref{Ellenbergob_subsection}). The resulting map \begin{equation*}\label{delta_3mod2_bdry_mp}\delta_3^{\mdl 2}:\rH^1(G_k, \pi/[\pi]_3^2) \rightarrow \rH^1(G_k,  [\pi]_3/[\pi]_4 ([\pi]_3)^2)\end{equation*} is defined as $\delta_3^{\mdl 2}$. The basis $\{ [[x,y],x], [[x,y],y] \}$ decomposes $\delta_3^{\mdl 2}$ into two obstructions $$\delta^2_{3,[[x,y],x]},\delta^2_{3,[[x,y],y]}: H^1(G_k, \pi/[\pi]^2_3) \rightarrow H^2(G_k, \Z/2)$$ which are compatible with the previously defined $\delta_{3,[[x,y],x]}$, $\delta_{3,[[x,y],y]}$ in the obvious manner. In words, $\delta_3^{\mdl 2}$ is  $\delta_3$ reduced mod $2$.

The obstruction $\delta_3^{\mdl 2}$ can also be constructed from a central extension extension of groups with an action of $G_k$. To see this, we recall certain well-known results on the lower (exponent $p$) central series of free groups: for a free group $F$, the successive quotient $[F]_n/[F]_{n+1}$ of the lower central series is isomorphic to the homogeneous degree $n$ component of the free Lie algebra on the same generators. The Lie basis theorem gives bases for $[F]_n/[F]_{n+1}$ explicitly via bases for the free Lie algebra \cite[Thm 5.8]{Magnus_Karrass_Solitar}. For the free group on $2$ generators $x$ and $y$, $$\{x,y \}, \{[x,y] \}, \{[[x,y],x],[[x,y],y] \}, \{[[[x,y],x],x],[[[x,y],y],y],[[[x,y],y],x]\}  $$ are bases for $n=1,2,3,4$. Results of \cite{Magnus_Karrass_Solitar} and \cite{Lazard} can be used to show that if $\beta_i$ is a basis for $[F]_i/[F]_{i+1}$ for $i=1,\ldots, n$ and $\beta_i^{p^{n-i}}$ denotes the set whose elements are the elements of $\beta_i$ raised to the $p^{n-i}$, then $$\beta_1^{p^{n-1}} \cup \beta_2^{p^{n-2}} \cup \beta_3^{p^{n-3}} \cup \ldots \cup \beta_n$$ is a basis for $[F]^p_n/[F]^p_{n+1}$, where $[F]_n^p$ denotes the $n^{th}$ subgroup of the lower exponent $p$ central series. Thus, as a $G_k$-module we have 
\begin{equation*}
[\pi]_3^p/[\pi]_4^p =  \Z/p\Z(3) \cdot [[x,y],x] \oplus  \Z/p\Z(3) \cdot [[x,y],y]  \qquad \qquad
\end{equation*}
\[
\oplus \Z/p\Z(2) \cdot [x,y]^2 \oplus \Z/p\Z(1) \cdot x^4 \oplus \Z/p\Z(1) \cdot y^4.
\] It follows that \begin{equation}\label{delta3mod2CE} 1 \rightarrow [\pi]_3/[\pi]_4 ([\pi]_3)^p \rightarrow \pi/[\pi]_4^p([\pi]_2^p)^p\rightarrow \pi/[\pi]_3^p \rightarrow 1 \end{equation} is an exact sequence, and $\delta_3^{\mdl 2}$ is also the boundary map in $G_k$ cohomology of \eqref{delta3mod2CE} for $p=2$. 

Since we view Ellenberg's obstructions as constraints for points on the Jacobian, we will define a \textit{lift} of a point of $\Jac (X)(k)$ to $\rH^1(G_k, \pi/[\pi]_3^2)$ and then define \[
\delta_3^{\mdl 2} \ :  \ \Ker \big(\delta_2^2 :  \Jac (\proj^1_{k}-\{0,1,\infty\})(k) \to \rH^2(G_k,[\pi]_2^2/[\pi]_3^2)\big) \to \{0, 1\}
\]which assigns to $(b,a)$ in  $k^* \times k^*$ the element $0$ if there exists a lift of $(b,a)$ such that $\delta_3^{\mdl 2} = 0$, and assigns to $(b,a)$ the element $1$ if there does not exist such a lift.  If $\delta_3^{\mdl 2} (b,a) \neq 0$, it follows that $(b,a)$ is not the image of a $k$ point or tangential point under the Abel-Jacobi map.

For $(b,a)$ in $k^* \times k^* = \Jac (\proj^1_{k}-\{0,1,\infty\})(k)$, let $(b,a)$ also denote the associated element of $\rH^1(G_k, \pi^{ab})$. For a characteristic closed subgroup $N < \pi$, a \textit{lift} of $(b,a)$ to $\rH^1(G_k,\pi/N)$ is an element whose image under
$$ 
 \rH^1(G_k, \pi/N) \rightarrow \rH^1(G_k, (\pi/N)^\ab)
 $$ 
equals the image of $(b,a)$ under 
$$
\Jac(X)(k) \to \rH^1(G_k, \pi^{\ab}) \rightarrow \rH^1(G_k, (\pi/N)^\ab).
$$

\label{liftstopi/[pi]^2_3}For example, the lifts of $(b,a)$ to $\rH^1(G_k, \pi/[\pi]_3^2)$ can be described as follows.
The fixed embeddings $\C \supset \overline{\Q} \subseteq \kbar$ and resulting identification $\pi = \langle x,y\rangle^\wedge$ give canonical identifications 
\[
(\pi/[\pi]_3^2)^\ab = \Z/4(1) \cdot x \times \Z/4(1) \cdot y  
\]
\[
 \Ker(\pi/[\pi]_3^2 \rightarrow (\pi/[\pi]_3^2)^\ab) = \Z/2(2)\cdot [x,y].
 \]
Choose fourth roots of $b$ and $a$. These choices give cocycles via the Kummer map $$b,a: G_{k} \rightarrow \Z/4(1),$$ and $g \mapsto y^{a(g)}x^{b(g)}$ represents $(b,a)$ in $\rH^1(G_k, (\pi/[\pi]_3^2)^\ab)$. The obstruction to lifting $(b,a)$ to $\rH^1(G_{k}, \pi/[\pi]_3^2)$ is $\delta_2^2(b,a) = b \cup a \in \rH^1(G_k, \Z/2(2))$.

For $(b,a)$ such that $\delta_2^{2} (b,a) = 0$, the lifts of $(b,a)$ are in bijection with the set of cochains $c \in C^1(G_k, \Z/2)$ such that $dc= -b \cup a$ (note the minus sign) up to coboundary by $$c \leftrightarrow (b,a)_c,$$ where $$(b,a)_c(g) = y^{a(g)}x^{b(g)}[x,y]^{c(g)}.$$ The set of these lifts is a $\rH^1(G_k,\Z/2(2))$ torsor. We could have equivalently considered the set of cochains $c$ such that $dc= -b \cup a$, as the $G_k$ action on $\Z/2(2)$ is trivial c.f. Corollary \ref{lifts(b,a)_c_of(b,a)}.

\begin{example}  \label{ex:b(1-b)to4}
(1)  Let $(b,a) \in \Jac(X)(k)$ be the image of a rational point or a rational tangential point of $X = \pmk$. For $x,y \in k^\ast$ then $(bx^4,ay^4)$ is unobstructed by $\delta^{\mdl 2}_2$ and $\delta_3^{\mdl 2}$ since $(b,a)$ is unobstructed and defines the same class in $\rH^1(G_k,\pi^\ab \otimes \Z/4\Z)$.

(2) Similarly, the mod $m$ obstructions $\delta_n^{\mdl m}$ coming from the lower $m$-central series of $\pi$ are $m$-adically continuous in the sense that the obstruction map $\delta_{n+1}^{\mdl m}$ is constant on cosets by $(k^\ast)^{m^n}$.

(3) $(b,(1-b)^{4})$ determines the same element of $\rH^1(G_K, \pi/([\pi]_2 \pi^4))$ as $(b,1)$ which is the image of a $k$ tangential point at $0$ (Lemma \ref{AbelJacobi(tgtpnt)}). Therefore $$\delta_3^{\mdl 2}(b,(1-b)^{4}) = 0.$$ This is an example of a point of the Jacobian unobstructed both by $\delta_2$ (Proposition \ref{some_pts_in_Kerdelta_2}) and $\delta_3^{\mdl 2}$.
\end{example}

We can also compute $\delta_3^{\mdl 2}(b,(1-b)^{4})$ directly using Proposition \ref{delta_3_p1minus3_cocycle}. We include the calculation. Let $[\pi]^m_{4}$ denote the subgroup of the lower exponent $m$ central series c.f. \ref{Ellenbergob_subsection}.

\begin{proposition}\label{(b,(1-b)^{m^2})lifts_H1(Gk,pi/pi^m_4)} The conjugacy class of the section of \begin{equation}\label{abm4sections} 1 \rightarrow \pi^{ab}/(\pi)^{m^3} \rightarrow \pi_1(\pmk, \01)/([\pi]_2 [\pi]^m_{4}) \rightarrow G_k \rightarrow 1\end{equation} determined by $(b,(1-b)^{m^2})$ lifts to a section of $$1 \rightarrow \pi/[\pi]^m_{4} \rightarrow \pi_1(\pmk, \01)/[\pi]^m_{4} \rightarrow G_k \rightarrow 1.$$\end{proposition}

\begin{proof}
Let $(b,a)= (b,(1-b)^{m^2})$. Choose compatible systems of $n^{th}$ roots of $b$ and $1-b$, giving cocycles $b,1-b: G_k \rightarrow \Zhat(1)$, and let $a: G_k \rightarrow \Zhat(1)$ be $m^2(1-b)$. Conjugacy classes of sections of (\ref{abm4sections}) are in bijection with $\rH^1(G_k, \pi^{ab}/(\pi)^{m^3})$ using $\01$ as the marked splitting. We claim that \begin{equation}\label{H1forba}g \mapsto y^{a(g)}x^{b(g)}\end{equation} defines a cocycle $G_k \rightarrow \pi/[\pi]_4^m,$  showing the proposition. Note that the reduction mod $m^2$ of $a$ is $0$, showing that $b \cup a = 0$ mod $m^2$, whence \eqref{H1forba} is a cocycle valued in $\pi/[\pi]_3^m$. By Proposition \ref{delta_3_p1minus3_cocycle} with $a=0$ mod $m$ and $c=0$, we have that $\delta_{3, [[x,y],x]}^m = 0$. An easy algebraic manipulation shows that ${\chi(g)a(h) + 1\choose 2}$ is $0$ mod $m^2$ for $m$ odd, and $0$ mod $m$ for any $m$, so the expression for $\delta_{3,[[x,y],y]}$ of Proposition \ref{delta_3_p1minus3_cocycle} with $a=c=0$ mod $m$ shows that $\delta_{3, [[x,y],y]}^m = 0$.
\qed \end{proof}

\subsection{Evaluating the $2$-nilpotent quotient of $\frak{f}$}\label{2nilquotf}
The cocycle $\frak{f}: G_k \rightarrow [\pi]_2$ in the description given in \eqref{G_k_action_F_2} of the Galois action on $\pi_1(\pmkbar, \01)$ records the monodromy of the standard path from $\01$ to $\10$. To evaluate $\delta_{[[x,y],y]}^2$, we will need to evaluate the $2$-nilpotent quotient $f : G_k \to \Zhat(2)[x,y]$ of $\frak{f}$ that was introduced in \eqref{fdef}, or more precisely its mod $2$ reduction.
Work of Anderson \cite{Anderson_hyperadelic_gamma}, Coleman \cite{Coleman_Gauss_sum}, Deligne, Ihara \cite[6.3 Thm~p.115]{Ihara_Braids_Gal_grps}, Kaneko, and Yukinari \cite{IKY} gives the equation 
\begin{equation}\label{f_is_(1-chi^2)/24}
f(\sigma) = \frac{1}{24}(\chi(\sigma)^2 - 1)
\end{equation}
where we recall that $\chi: G_k \rightarrow \Zhat^*$ denotes the cyclotomic character.  For the convenience of the reader, we introduce enough of the notation of \cite{Ihara_Braids_Gal_grps} to check (\ref{f_is_(1-chi^2)/24}) from the statement given in loc. cit. 6.3 Thm~p.115.

Let $\Zhat \langle \langle \xi, \eta \rangle \rangle$ denote the non-commutative power series algebra in two variables over $\Zhat$. The Magnus embedding of the free group on two generators in $\Zhat \langle \langle \xi, \eta \rangle \rangle$ gives rise to an injective  map 
$$
\mathcal{M} \ : \ \pi\cong \langle x,y\rangle^{\wedge} \inj \Zhat \langle \langle \xi, \eta \rangle \rangle
$$ 
defined by 
$$
\mathcal{M}(x) = 1 + \xi 
$$ 
$$
\mathcal{M}(y) = 1 + \eta
$$
By \cite{Magnus_Karrass_Solitar} Cor 5.7, $\mathcal{M}$ takes $[\pi]_n$ to elements of the form $1+ \sum_{m \geq n} u_m $, where $u_m$ is homogeneous of degree $m$. By \cite[\S 5.5]{Magnus_Karrass_Solitar} Lemma 5.4, for any $j\in \Zhat$, $$\mathcal{M}([x,y]^j) = 1 + j(\xi \eta - \eta \xi) + \sum_{m > 2} u_m $$ where $u_m$ is some homogeneous element of degree $m$ (depending on $j$). More generally, to lowest order in $\Zhat \langle \langle \xi, \eta \rangle \rangle$, $\mathcal{M}$ takes the commutator of the group $\pi$ to the Lie bracket of the associative algebra $\Zhat \langle \langle \xi, \eta \rangle \rangle$, in the manner made precise by loc. cit. \S 5.5 Lemma 5.4 (7).
Thus
\begin{equation}\label{frakfloworder}
\frak{f}(\sigma) = 1 + f(\sigma)(\xi \eta - \eta \xi)+ \mathcal{O}(3)
\end{equation} 
where $\mathcal{O}(3)$ is a sum of monomials of degree $\geq 3$.

As a $\Zhat$ module, $\Zhat \langle \langle \xi, \eta \rangle \rangle$ is the direct sum $$\Zhat \langle \langle \xi, \eta \rangle \rangle \cong \Zhat \oplus  \Zhat \langle \langle \xi, \eta \rangle \rangle \xi \oplus \Zhat \langle \langle \xi, \eta \rangle \rangle \eta.$$ Define $\psi: G_k \rightarrow \Zhat \langle \langle \xi, \eta \rangle \rangle$ to be the projection of $(\frak{f})^{-1}$ onto the direct summand $\Zhat \oplus  \Zhat \langle \langle \xi, \eta \rangle \rangle \xi$ i.e. $$(\frak{f}(\sigma))^{-1} = 1+ a_1 \xi + a_2 \eta $$ $$\psi(\sigma) = 1 + a_1 \xi$$ By (\ref{frakfloworder}), we have $$(\frak{f}(\sigma))^{-1} = 1 - f(\sigma)(\xi \eta - \eta \xi)+ \mathcal{O}(3),$$ so $f(\sigma)$ is the coefficient of $\eta \xi$ in $\psi(\sigma)$. As the only degree $2$ terms that $\psi(\sigma)$ can contain are $\Zhat$ linear combinations of $\eta \xi$ and $\xi^2$, the cocycle $f$ is determined by the degree $2$ terms of the projection of $\psi(\sigma)$ to the commutative power series ring. More explicitly, let $\psi^{ab}: G_k \rightarrow \Zhat[[ \xi, \eta]]$ denote the composition of $\psi$ with the quotient $$\Zhat \langle \langle \xi, \eta \rangle \rangle \rightarrow \Zhat [[ \xi, \eta]]$$ where $\Zhat [[ \xi, \eta]]$ denotes the commutative power series ring. Then $$\psi^{ab}(\sigma) = 1 + f(\sigma) \eta \xi + r$$ where $r$ is a sum a monomial of the form $b \xi^2$ with $b \in \Zhat$ and monomials of degree greater than one.

The formula  in \cite[6.3 Thm~p.115]{Ihara_Braids_Gal_grps} expresses $\psi^{ab}(\sigma)$ in terms of the Bernoulli numbers and the variables $X = \operatorname{log}(1+ \xi)$, and $Y = \operatorname{log}(1+ \eta)$. This formula gives $$\psi^{ab}(\sigma) = 1 - \frac{1}{2} b_2 (1 - \chi(\sigma)^2) \eta \xi + \mathcal{O}(3)$$ where $\mathcal{O}(3)$ is a sum of monomial terms in the variables $\eta$,$\xi$ of degree $\geq 3$, and $b_2 = \frac{1}{12}$. This implies (\ref{f_is_(1-chi^2)/24}).

\medskip

We denote the mod $2$ reduction of $f$ by
\[
\overline{f} \ : \ G_k \to \Zhat(2) \to \Z/2Z(2) = \Z/2\Z. 
\]

\begin{lemma} \label{frakf2nilmod2} 
The class represented by $\overline{f}$ in $\rH^1(G_k,\Z/2\Z)$ is the class of $2$ under the Kummer map.
\end{lemma}

\begin{proof} We may assume that $k=\Q$ by functoriality. The value of $\frac{1}{24}(1 - \chi(\sigma)^2)$ mod $2$ is determined by $1 - \chi(\sigma)^2$ mod $48$, which in turn is determined by $\chi(\sigma)$ mod $24$. A direct check shows that $\overline{f}(\sigma)=1$ when $\chi(\sigma)$ is $\pm 3$ mod $8$, and $\overline{f}(\sigma)=0$ otherwise. Thus $\overline{f}$ corresponds to the quadratic extension $k \subset k(\zeta_8 + \zeta_8^{-1}) = k(\sqrt{2})$ inside $k \subset k(\zeta_8)$.
\end{proof}

\subsection{Local $3$-nilpotent obstructions mod $2$ at $\R$}\label{delta_3mod2R} 

We compute $\delta_3^{\mdl 2}$ for $k=\R$. For a point $(b,a)$ of $\R^\ast \times \R^\ast \cong \Jac(\proj_{\R}^1-\{0,1,\infty\})(\R),$ we have the associated element of $\rH^1(G_\R, \pi^{ab})$. Recall the notation $(b,a)_c$ and characterization of the lifts of this element to $\rH^1(G_\R, \pi/[\pi]_3^2)$ of  \ref{liftstopi/[pi]^2_3}. Recall as well that $\{-1\}$ in $C^1(G_k, \Z/2)$ denotes the image of $-1$ under the Kummer map, and note that given $c$ in $C^1(G_{\R}, \Z/2(2))$ such that $Dc = -b \cup a$, we have that $c + \{-1\}$ is another cochain such that $D(c + \{-1\}) = -b \cup a$. Thus $(b,a)_c$ and $(b,a)_{c + \{-1\}}$ are two lifts of $(b,a)$ to $\rH^1(G_\R, \pi/[\pi]_3^2)$, and they are the only two as $\rH^1(G_\R, \Z/2(2)) = \Z/2$.

\begin{proposition} \label{Prop_delta_3_2R}
For any $(b,a)_c$ in $\rH^1(G_\R, \pi/[\pi]_3^2)$, either 
$$
\delta_3^{\mdl 2}(b,a)_c = 0
\quad \text{ or } \quad 
\delta_3^{\mdl 2} (b,a)_{c + \{-1\}} = 0.
$$
\end{proposition}
\begin{proof}
Since $\overline{f} = 2 \in \rH^1(G_{\R},\Z/2\Z)$ vanishes (Lemma~\ref{frakf2nilmod2}), it suffices by 
Theorem~\ref{delta_3_Massey_prod_8_10} to show that the triple Massey products 
\[
\delta_{3,[[x,y],x]}^{\mdl 2}(b,a) = \langle \{-b\}, b, a \rangle  \in \rH^2(G_{\R},\Z/2\Z)
\]
\[
\delta_{3,[[x,y],y]}^{\mdl 2}(b,a) = \langle \{-a\}, a, b \rangle - \overline{f} \cup b =   \langle \{-a\}, a, b \rangle \in \rH^2(G_{\R},\Z/2\Z)
\]
admit the value $0$ for a compatible choice of the defining systems as in Remark \ref{Massey_Thm_Remark} (ii).  
Changing $c$ by a $1$-cocycle $\epsilon \in C^1(G_{\R},\Z/2\Z)$ has the effect by Proposition~\ref{8/10_delta_3_cocycle_form} that 
\begin{equation*}
\delta_3^{\mdl 2}(b,a)_{c} - \delta_3^{\mdl 2}(b,a)_{c+\epsilon} = 
\{-b\} \cup \epsilon \cdot [[x,y],x]  + \{-a\} \cup \epsilon \cdot [[x,y],y].
\end{equation*}
Since $a \cup b = 0$ we conclude that at most one of $a,b$ is negative.

\textit{Case $a,b>0$}: we may choose trivial defining systems (i.e. the defining system for $\langle \{-b\}, b, a \rangle$ is $\{ {b \choose 2}, 0  \} = \{0,0\}$ and the defining system for $ \langle \{-a\}, a, b \rangle$ is  $\{{a \choose 2}, ab + 0 \} = \{0, 0 \}$) so the obstruction vanishes.

\textit{Case $a>0$ and $b < 0$}: regardless of the defining system, the Massey product 
$\langle \{-b\}, b, a \rangle$ always vanishes. The other Massey product can be adjusted if necessary by $\epsilon = -1$ since $\{-1\} \cup \{-1\}$ generates $\rH^2(G_{\R},\Z/2\Z)$.

\textit{Case $a<0$ and $b>0$}: regardless of the defining system, the Massey product 
$\langle \{-a\}, a, b \rangle$ always vanishes. The other Massey product can be adjusted if necessary by $\epsilon = -1$ since $\{-1\} \cup \{-1\}$ generates $\rH^2(G_{\R},\Z/2\Z)$.
\qed
\end{proof}

\subsection{Local $3$-nilpotent obstructions mod $2$ above odd primes}\label{evaluation_delta32p} 

Let $k$ be a number field embedded into $\C$. Let $p \in \Z$ be an odd prime, $k_v$ the completion of $k$ at a prime $v$ above $p$, and choose an embedding $\overline{\Q} \subset \overline{k_v}$, where $\overline{\Q}$ is the algebraic closure of $\Q$ in $\C$. 

We compute $\delta_3^{(\mdl 2,v)}$ for all $k_v$-points of the Jacobian of $\proj^1_{k_v} - \{0,1,\infty\}$ in the kernel of $\delta_2^{(\mdl 2,v)}$. Note that \ref{cup_prod_table_Q_p} and 
\[
\delta_2^{\mdl 2,v}(b,a) = a \cup b \in \rH^2(G_{k_v},\Z/2\Z(2))
\]
characterize this kernel, and recall the notation $(b,a)_c$ for lifts of $(b,a) \in \Ker \delta_2^{\mdl 2,v}$ to $\rH^1(G_{k_v}, \pi/[\pi]^2_3)$ given in \ref{liftstopi/[pi]^2_3}.

\begin{lemma}\label{rank<2} Let $(b,a)$ in $\Jac(\proj^1_{k_v} - \{0,1,\infty\})(k_v)$ be in the kernel of $\delta_2^{\mdl 2,v}$ and let $(b,a)_c$ be a lift of $(b,a)$ to $H^1(G_{k_v}, \pi/[\pi]_3^2)$. 
\begin{enumerate}\item \label{loc_obst_-b=0}  If $\{-b\} = 0$ in $\rH^1(G_{k_v}, \Z/2)$, then $\delta_{3,[[x,y],x]}^{(2,v)} (b,a)_{c} = {b \choose 2} \cup a$ is independent of $c$. 
\item \label{loc_obst_-a=0} If $\{-a\} = 0$ in $\rH^1(G_{k_v}, \Z/2)$, then $\delta_{3,[[x,y],y]}^{(2,v)} (b,a)_{c} = {a \choose 2} \cup b + \overline{f} \cup a$ is independent of $c$. 
\item \label{loc_obst_b=a} If $\{b\} = \{a\}$ in $\rH^1(G_{k_v}, \Z/2)$, then $$\delta_{3,[[x,y],x]}^{(2,v)} (b,a)_{c} + \delta_{3,[[x,y],y]}^{(2,v)} (b,a)_{c} = ({a \choose 2} + {b \choose 2} + \overline{f}) \cup a $$ is independent of $c$.
\end{enumerate} Otherwise $$\delta_3^{(\mdl 2,v)} (b,a)_c = 0.$$
\end{lemma}

\begin{proof}
Changing the lift is equivalent to adding a $1$-cocycle $\epsilon \in C^1(G_{k_v},\Z/2\Z)$ to $c$ with the effect by Proposition~\ref{8/10_delta_3_cocycle_form} that 
\begin{equation} \label{eq:shiftMassey}
\delta_3^{\mdl 2}(a,b)_{c} - \delta_3^{\mdl 2}(a,b)_{c+\epsilon} = 
\{-b\} \cup \epsilon \cdot [[x,y],x]  + \{-a\} \cup \epsilon \cdot [[x,y],y].
\end{equation}
We abbreviate the differences componentwise by introducing the notation
$$
\Delta_{3,[[x,y],x]}^{2} (b,a,z)= \{-b\} \cup z,
$$
$$
\Delta_{3,[[x,y],y]}^{2} (b,a,z) = \{-a\} \cup z. 
$$
Using Proposition \ref{8/10_delta_3_cocycle_form}, case (\ref{loc_obst_-b=0}) and (\ref{loc_obst_-a=0}) are now immediate. In case (\ref{loc_obst_b=a}), note that Proposition \ref{8/10_delta_3_cocycle_form} implies that \[ \delta_{3,[[x,y],x]}^{(2,v)} (b,a)_{c} + \delta_{3,[[x,y],y]}^{(2,v)} (b,a)_{c} = (a + \frac{\chi -1}{2}) \cup (ab) + {b \choose 2} \cup a + {a \choose 2} \cup b + \overline{f} \cup a\] Since $a=b$, we have $(ab) = a^2= a$, so $$(a + \frac{\chi -1}{2}) \cup (ab) = (a + \frac{\chi -1}{2}) \cup a =0.$$ This equation and $a=b$ show case (\ref{loc_obst_b=a}).

If the map 
\begin{equation} \label{eq:varianceMassey}
\Delta \ : \ \rH^1(G_{k_v},\Z/2\Z) \to \rH^2(G_{k_v},\Z/2\Z) \oplus \rH^2(G_{k_v},\Z/2\Z)
\end{equation}
\[
z \mapsto \big(\Delta_{3,[[x,y],x]}^{2} (b,a,z), \Delta_{3,[[x,y],y]}^{2} (b,a,z)\big) = \big(\{-b\} \cup z,
\{-a\} \cup z\big)
\]
is surjective, namely by the nondegeneracy of the cup product pairing, if $\{-a\},\{-b\}$ forms a basis of 
$\rH^1(G_{k_v},\Z/2\Z)$, then for a suitable choice of correction term $z$, the corresponding $c$ will lead to a vanishing
\[
\delta_3^{\mdl 2} ((b,a)_c) = 0.
\]
Since $\rH^1(G_{k_v},\Z/2\Z) \cong \Z/2\Z \oplus \Z/2\Z$, the elements $\{-a\},\{-b\}$ form a basis unless at least one of cases (\ref{loc_obst_-b=0}), (\ref{loc_obst_-a=0}), or (\ref{loc_obst_b=a}) holds.
\qed \end{proof}

To compute when there is a local obstruction as in case (\ref{loc_obst_-b=0}) or (\ref{loc_obst_-a=0}), we need to understand the cochain $${a \choose 2}: G_{k_v} \rightarrow \Z/2$$ when $\{-a\}$ is $0$ in $C^1(G_{k_v}, \Z/2)$ cf. Examples \ref{Dbchoose2} and \ref{choose_rest_sqrt}.

\begin{lemma}\label{achoose2_when_-a=0}  Let $K$ be a field of characteristic $\not=2$, and let $\zeta_4$ be a primitive fourth root of unity in a fixed algebraic closure of $K$. Let $-x\in (K^\ast)^2$ be a square, and choose a fourth root $\sqrt[4]{x}$ of $x$ giving a Kummer cocycle $x : G_K \to \Z/4\Z(1)$ and the cochain ${ x \choose 2} : G_K \to \Z/2\Z$ obtained by the identification $\Z/4\Z(1) = \Z/4\Z(\chi)$ using $\zeta_4$. Then there is an equality of cocycles \[
{x \choose 2} = \{2 \sqrt{-x} \} : G_K \to \Z/2\Z
\] where $\sqrt{-x} = \zeta_4 (\sqrt[4]{x})^2$.
\end{lemma}

\begin{proof}
Note that squaring $(1+ \zeta_4)\sqrt[4]{x}$ gives $2 \sqrt{-x}$, so letting $\eta = 2 \sqrt{-x} \in K$, we see that $K(\sqrt[4]{x}, \zeta_4)=K(\sqrt{\eta}, \zeta_4)$. Both cochains ${ x \choose 2}, \{2 \sqrt{-x} \} : G_K \to \Z/2\Z$ factor through $\Gal(K(\sqrt[4]{x}, \zeta_4)/K)$, and there are four possibilities for the action of $g \in G_K$ on $\sqrt{\eta}$ and $\zeta_4$. It is enough to check that ${ x \choose 2}(g)$ and $\{\eta \}(g)$ agree in each case: 
\[
\begin{array}{c|c|c|c}
\quad g(\sqrt{\eta}) \quad & \quad g(\zeta_4) \quad & \ x(g) \ & \quad {x(g) \choose 2} \quad \\[1ex] \hline
\sqrt{\eta} & \zeta_4 &  (\zeta_4)^0 &   0 \\
\sqrt{\eta} & (\zeta_4)^3  & (\zeta_4)^1 & 0\\
-\sqrt{\eta} & \zeta_4 & (\zeta_4)^2  & 1 \\
-\sqrt{\eta} & (\zeta_4)^3  &  (\zeta_4)^3 & 1 \\
\end{array}
\]
This shows the claim ${x \choose 2} = \{\eta\} = \{ 2 \sqrt{-x} \}$.
\qed \end{proof}

To compute when there is a local obstruction as in case (\ref{loc_obst_b=a}), we need to understand the cochain $${a \choose 2} + {b \choose 2}: G_{k_v} \rightarrow \Z/2$$ when $a = b$ in $C^1(G_{k_v}, \Z/2)$.

\begin{lemma}\label{achoose2+bchoose2_when_a=b} Let $K$ be a field of characteristic $\not=2$, and let $\zeta_4$ be a primitive fourth root of unity in a fixed algebraic closure of $K$. Let $a,b$ be non-zero elements of $K$ such that $a/b$ is a square $a/b \in (K^\ast)^2$. Choose fourth roots of both, and let $a,b: G_{K} \rightarrow \Z/4(1)$ denote the corresponding cocyles. Then $${a \choose 2} + {b \choose 2}: G_{K} \rightarrow \Z/2$$ equals the cocycle $$ \{\sqrt{b}/\sqrt{a}\}: G_{K} \rightarrow \Z/2, $$ where $\sqrt{b}$, $\sqrt{a}$ are defined as the squares  of the chosen fourth roots of $b$, $a$ respectively.\end{lemma}

\begin{proof}
For any two elements $d_1,d_2$ in $\Z/4$, direct calculation shows that in $\Z/2$ $${d_1 + d_2 \choose 2} - {d_1 \choose 2} - {d_2 \choose 2} = d_1 d_2.$$

Therefore for all $g$ in $G_K$, $$ {a(g) \choose 2}+ {b(g) \choose 2} = {a(g) + b(g) \choose 2} - a(g)b(g)$$ Since mod $2$, $a(g) = b(g)$, we have that $a(g)b(g)=a(g)$ mod $2$. Therefore $$ {a(g) \choose 2}+ {b(g) \choose 2} = {a(g) + b(g) \choose 2} - a(g)$$

Since $b/a$ is a square in $K$, $ab$ is also a square in $K$. Let $\sqrt{a}$,$\sqrt{b}$ denote the squares of our chosen fourth roots of $a$ and $b$ respectively. 

By Lemma \ref{choose_rest_sqrt}, and because $ab$ is a square in $K$, $$ g \mapsto {a(g) + b(g) \choose 2}$$ equals $\{ \sqrt{a}\sqrt{b} \}$. Therefore $$ {a \choose 2}+ {b \choose 2} = \{\sqrt{a}\sqrt{b} \} - a.$$ Since $\{\sqrt{a}\sqrt{b} \} - a= \{\sqrt{b}/\sqrt{a} \}$, the lemma is shown.
\qed \end{proof}

The previous results combine to give necessary and sufficient conditions for $\delta_3^{(\mdl 2,v)}$ to obstruct a $k_v$-point of the Jacobian in the kernel of $\delta_2^{(\mdl 2,v)}$ from lying on the curve (i.e. from being the image of a $k_v$-point or tangential point of $X$). For any given point $(b,a)$ of $\Jac X$, these conditions are easy to verify using \ref{cup_prod_table_Q_p}.

\begin{theorem}\label{Prop_ploc_mod2_3nil_obs} 
Let $k_v$ be the completion of a number field $k$ at a place above an odd prime\hidden{, with fixed $\C \supset \overline{k} \subset \overline{k_v}$}. For $(b,a) \in \Jac(\proj^1_{k_v} - \{0,1,\infty\})(k_v)$ such that $\delta_2^{(\mdl 2,v)}(b,a)=0$, we have
$\delta_3^{\mdl 2} (b,a) \neq 0$ if and only if one of the following holds.
\begin{description}
\item[(i)] \label{Prop_loc_obs_b=-1} 
$-b \in (k_v^\ast)^2$ and $\{2 \sqrt{-b}\} \cup a \neq 0$.
\item[(ii)]  \label{Prop_loc_obs_a=-1} 
$-a \in (k_v^\ast)^2$ and $\{2 \sqrt{-a}\} \cup b + \{2\} \cup a \neq 0$
\item[(iii)]  \label{Prop_loc_obs_a=b} 
$ab \in (k_v^\ast)^2$ and $\{ 2 \sqrt{b} \sqrt{a} \} \cup a \neq 0$\end{description}
\end{theorem}

\begin{remark}\label{remarks_on_Prop_ploc_mod2_3nil_obs} In case (i), the notation $\sqrt{-b}$ denotes either square root of $-b$, both of which are in $k_v$. The expression $\{2 \sqrt{-b}\} \cup a$ denotes the corresponding element of $\rH^2(G_{k_v},\Z/2\Z)$, which is independent of the choice of square root because $$\{-1\} \cup a = \{-b\} \cup a - b \cup a = 0.$$ Similar remarks hold in cases (ii) and (iii).

Note that obstructions such as $\{2 \sqrt{-b}\} \cup a$ look as though they are naturally elements of $\rH^2(G_{\Q_p},\Z/2\Z(2))$, but $\delta_{3,[[x,y],x]}^{(2,p)} (b,a)_{c}$ is in $\rH^2(G_{\Q_p},\Z/2\Z(3))$. The shift in weight happened in Lemmas \ref{choose_rest_sqrt}, \ref{achoose2_when_-a=0}, and \ref{achoose2+bchoose2_when_a=b}.\end{remark}

\begin{proof}
Fix $\C \supset \overline{k} \subset \overline{k_v}$. It is sufficient to show that in Lemma \ref{rank<2} case (i) (ii) (iii) respectively, we have $${b \choose 2} \cup a = \{2 \sqrt{-b}\} \cup a,$$ $${a \choose 2} \cup b + \overline{f} \cup a =  \{2 \sqrt{-a}\} \cup b + \{2\} \cup a,$$ $$({a \choose 2} + {b \choose 2} + \overline{f}) \cup a = \{ 2 \sqrt{b} \sqrt{a} \} \cup a.$$ For cases (i) and (ii), this follows immediately from Lemmas \ref{achoose2_when_-a=0} and \ref{frakf2nilmod2}. In case (iii), Proposition \ref{achoose2+bchoose2_when_a=b} and \ref{frakf2nilmod2} show that $$ ({a \choose 2} + {b \choose 2} + \overline{f}) \cup a = \{2 \sqrt{b}/\sqrt{a}\} \cup a.$$ Since  $\sqrt{b}/\sqrt{a}$ is in $k_v$, we have that $a \sqrt{b}/\sqrt{a} = \sqrt{b}\sqrt{a}$ is in $k_v$. As $a \cup b = 0$ and $a = b$, it follows that $ \{2 \sqrt{b}/\sqrt{a}\} \cup a$ is also equal to $ \{2 \sqrt{b}/\sqrt{a}\} \cup a + a \cup a$, which in turn equals $\{ 2 \sqrt{b} \sqrt{a} \} \cup a$.
\qed \end{proof}

\begin{corollary}\label{local_mod2_delta2_3_cor_version}
Let $(b,a)$ be a rational point of $\Jac(\pmQ)$ such that $(b,a)$ is in $(\Z - \{0\}) \times (\Z - \{0\})$ and $p$ divides $ab$ exactly once. Then
 \begin{enumerate}
 \item \label{localdelta2genp} $\delta_2^{(\mdl 2,p)} (b,a) = 0 \iff a+b$ is a square mod $p$ 
 \item \label{localdelta3genp} When (\ref{localdelta2genp}) holds, $\delta_3^{(\mdl 2,p)} (b,a) = 0 \iff a+b$ is a fourth power mod $p$
 \end{enumerate}
\end{corollary}

\begin{remark}\label{nilscp1rm}
(1) Note that under the hypotheses of Corollary~\ref{local_mod2_delta2_3_cor_version}, the condition that $a+b$ is congruent to a $(2^n)^{th}$ power mod $p$ is equivalent to the condition that whichever of $a$ or $b$ not divisible by $p$ is a $(2^n)^{th}$ power mod $p$.

(2) For the points of the Jacobian satisfying the conditions described in its statement, Corollary~\ref{local_mod2_delta2_3_cor_version} computes $\delta_3^{(\mdl 2,p)}$ and $\delta_2^{(\mdl 2,p)}$ in terms of congruence conditions mod $p$. These congruence conditions allow us to see that $\delta_3^{(\mdl 2,p)}$ and $\delta_2^{(\mdl 2,p)}$ vanish on the points and tangential points of the curve which satisfy the hypotheses of the Corollary. Namely, by \eqref{aj(x)_is_(x,1-x)} and Lemmas~\ref{AbelJacobi(tgtpnt)}, the image of $\proj^1 - \{0,1,\infty\}$ and its tangential points is the set of $(b,a)$ such that $a+b=0$, $a+b = 1$, $a=1$, or $b=1$. As $0$ and $1$ are fourth powers mod every prime, we see the vanishing of Ellenberg's obstructions on the points of the curve. 

(3) It is tempting to hope that under certain hypotheses
\begin{itemize}
\item $\delta_n^{(\mdl 2, p)}(b,a) = 0 \iff a+b $ is a $2^{n-1}$ power mod $p$.
\end{itemize}\hidden{Hypotheses are necessary, even discarding the case $a=1$ or $b=1$: note that if $a=p^2$ and $b$ is a square mod $p$ but not a fourth power, $\delta_3^{(\mdl 2, p)}(b,a) = 0$.} Since $0$ and $1$ are the only integers which are $2^{n-1}$ powers for every $n$, mod every prime, such a result could show that the nilpotent completion of $\pi$ determines the points and tangential points of $\proj^1 - \{0,1,\infty\}$ from those of the Jacobian.\hidden{Up to the difference between $H^1(G_{\Q}, \pi/[\pi]_2)$ and the points of the Jacobian--and this difference is a profinite completion,} This is a mod $2$, pro-nilpotent section conjecture for $\proj^1 - \{0,1,\infty\}$ cf. \ref{2nilseccon}.\end{remark}

\begin{proof}
To prove (\ref{localdelta2genp}): note that by hypothesis, exactly one of $b$ and $a$ equals $\mathfrak{p}$ or $u + \mathfrak{p}$ in $\rH^1(G_{\Q_p}, \Z/2(1))$, where the notation $\mathfrak{p}$,$u$ is as defined in \ref{cup_prod_table_Q_p}. By \ref{cup_prod_table_Q_p}, it follows that the other is $0$ in $\rH^1(G_{\Q_p}, \Z/2(1))$ if and only if $b \cup a$ vanishes. By Hensel's Lemma, this is equivalent to the other being a square mod $p$, which in turn is equivalent to the condition that $a+b$ is a square mod $p$.

To prove (\ref{localdelta3genp}): exactly one of $b$ and $a$ is not divisible by $p$. Call this element $r$. The only case listed in Theorem \ref{Prop_ploc_mod2_3nil_obs} that can hold is $\{-r\} = 0$ in $\rH^1(G_{\Q_p}, \Z/2\Z(1))$. By (\ref{localdelta2genp}), we have that $r$ is a square mod $p$, whence $\{-r\} = \{-1\}$ in $\rH^1(G_{\Q_p}, \Z/2\Z(1))$. Note also that if $r=a$, then $\overline{f} \cup a = \{2\} \cup a = 0$, as neither $2$ nor $a$ is divisible by $p$. Therefore, $\delta_3^{(\mdl 2,p)} (b,a) \neq 0$ if and only if $p = 1$ mod $4$, and $\{2 \sqrt{-r}\} \cup p \neq 0$. 

For $p = 1$ mod $4$, and $r$ and $-r$ squares mod $p$,  $$\{2 \sqrt{-r}\} \cup p = \{ \sqrt{r} \} \cup p$$ in $\rH^1(G_{\Q_p}, \Z/2\Z(1))$ where either square root of $r$ or $-r$ in $\Q_p$ can be chosen. To see this: note that since $-1$ is a square, it is clear that changing the square root has no effect. Note that $(1+ \xi_4)^2 = 2 \xi_4$, for $\xi_4$ a primitive fourth root of unity in $\Q_p$, from which it follows that $\{ \sqrt{r} \} = \{ \sqrt{r} (1+ \xi_4)^2 \} = \{2 \sqrt{-r} \}$. (In the last equality $\sqrt{-r}$ is $\xi_4 \sqrt{r}$, but we are free to choose either square root to see the claimed equality.)

Thus, $\delta_3^{(\mdl 2,p)} (b,a) \neq 0$ if and only if $p = 1$ mod $4$, and $\{ \sqrt{r} \} \cup p \neq 0$ in $\rH^1(G_{\Q_p}, \Z/2\Z(1))$. Note that $\{ \sqrt{r} \} \cup p \neq 0$ if and only if $\{ \sqrt{r} \} \neq 0$, since $r$ is not divisible by $p$. The condition $p = 1$ mod $4$ and $\{ \sqrt{r} \} \neq 0$ is equivalent to the condition $p = 1$ mod $4$ and $r$ is not a fourth power mod $p$. Since $r$ is a square mod $p$, the condition that $r$ is not a fourth power implies that $p = 1$ mod $4$. Thus, $\delta_3^{(\mdl 2,p)} (b,a) \neq 0$ if and only if $r$ is not a fourth power mod $p$. This last condition is equivalent to $a+b$ is not a fourth power mod $p$.
\qed \end{proof}

\begin{remark} The proof of Corollary~\ref{local_mod2_delta2_3_cor_version} only uses the computation of $\overline{f}$ given in \ref{frakf2nilmod2} to ensure that $f \in \{ 0, \frak{u}\} \subset \rH^1(G_{\Q_p}, \Z/2)$.\end{remark}

\begin{definition}
For an obstruction $\delta'$ which is defined on the vanishing locus of an obstruction $\delta$, we say that \textbf{$\delta'$ is not redundant with $\delta$} if $\delta'$ does not vanish identically.
\end{definition}

\begin{example} \label{delta32nontriv} We compare the obstruction $\delta_3^{\mdl 2}$ with $\delta_2$ for $k=\Q$.

(1) As $4 = (-1) + 5$ is a square  but not a fourth power mod $5$, Corollary~\ref{local_mod2_delta2_3_cor_version} implies that $\delta_2^{(\mdl 2,5)}(-1,5) = 0$, and 
$$
\delta_3^{(\mdl 2,5)}(-1,5) \neq 0.
$$ 
In other words, the $3$-nilpotent obstruction $\delta_3^{(\mdl 2,p)}$ is not redundant with the $2$-nilpotent obstruction $\delta_2^{(\mdl 2,p)}$. 

(2) In fact, it is easy to check that $\{-1\} \cup \{5\} = 0$ in $\rH^2(G_{\Q}, \Z/2\Z)$ since
\[
\{-1\} \cup \{5\} = \{-1\} \cup \{5\} + \{1-5\} \cup \{5\} = \{4\} \cup \{5\} = 2\cdot(\{2\} \cup \{5\}) = 0.
\]
Alternatively, $\{-1\} \cup \{5\} = 0$ in $\rH^2(G_{\Q}, \Z/2)$ because the Brauer-Severi variety $$-u^2 + 5 v^2 = w^2 $$ has the rational point $[u,v,w] = [1,1,2]$. Thus, $\delta_3^{(\mdl 2,p)}$ is not redundant with the global obstruction $\delta_2^{\mdl 2}$.  It also follows that the global $3$-nilpotent obstruction $\delta_3^{\mdl 2}$ is not redundant with the global $2$-nilpotent obstruction $\delta_2^{\mdl 2}$. 

(3) One can ask whether $\delta_3^{(\mdl 2,p)}$ is redundant with the global obstruction $\delta_2$ for $k = \Q$. The tame symbol at $p$ 
$$ 
b \otimes a \mapsto (b,a)_p= (-1)^{v_p(b) v_p(a)} \frac{b^{v_p(a)}}{a^{v_p(b)}} \in  \bF_p^\ast
$$ 
vanishes on any $(b,a)$ such that $\delta_2 (b,a) = 0$. In particular, given $b,a \in \Z$ such that $p$ divides $ab \not=0$ exactly once, we will have that 
$$
\frac{b^{v_p(a)}}{a^{v_p(b)}} = 1 \mod p
$$ 
and that $\frac{b^{v_p(a)}}{a^{v_p(b)}}$ equals either $b$ or $1/a$ depending on which of $b$ or $a$ is divisible by $p$. In particular, $a+b = 1$ mod $p$, and thus, Corollary~\ref{local_mod2_delta2_3_cor_version} does not show that $\delta_3$ is not redundant with $\delta_2$ for $k = \Q$.
\end{example}

The points $(b,a)$ of $\Jac (\pmQ)(\Q) = \Q^{\ast} \times \Q^{\ast}$ considered in Corollary~\ref{local_mod2_delta2_3_cor_version} and satisfying $\delta_2^{(\mdl 2,p)}(b,a) = 0$ have the property that at any finite prime p, either $b$ or $a$ determines the $0$ element of $$C^1(G_{\Q_p}, \Z/2) \cong \rH^1(G_{\Q_p}, \Z/2).$$ Thus a lift $(b,a)_c$ of $(b,a)$ to a class of $\rH^1(G_{\Q_p}, \pi /[\pi]_3^2)$ is such that $c$ is a cocycle, as opposed to a cochain. By Theorem \ref{delta_3_Massey_prod_8_10}, Corollary~\ref{local_mod2_delta2_3_cor_version} consists of evaluations of Massey products where certain cup products of cochains are not only coboundaries, but $0$ as a cochains. Indeed, a direct proof of Corollary~\ref{local_mod2_delta2_3_cor_version} can be given along these lines, although the methods involved are not sufficiently different from those of Theorem \ref{Prop_ploc_mod2_3nil_obs} to merit inclusion.

Let $p$ vary through the odd primes, and let $m$ vary through the positive integers. The points $((-p)^{2m+1},p)$ satisfy $\delta_2 = 0$ by Proposition~\ref{some_pts_in_Kerdelta_2}, but both $(-p)^{2m+1}$ and $p$ determine non-zero elements of $C^1(G_{\Q_p}, \Z/2)$ via the Kummer map, unlike the examples computed via Corollary~\ref{local_mod2_delta2_3_cor_version}. 
Theorem~\ref{Prop_ploc_mod2_3nil_obs} allows us to evaluate 
$\delta_3^{(\operatorname{mod} 2,p)}$ on these points.

\begin{example}  \label{compcalex}  
Let $p$ be an odd prime and $m$ a positive integer. Then $\delta_3^{(\operatorname{mod} 2,p)}$ vanishes on the following rational points of $\Jac (\pmQ)$ 
\begin{description}
\item[(1)] $(-p^{2m+1}, p)$,
\item[(2)] $(p^{2m},p)$.
\end{description}
For $(-p^{2m+1}, p)$ we show that neither case (i)-(iii) of Theorem~\ref{Prop_ploc_mod2_3nil_obs} holds. Note that $\{p^{2m+1} \}$ (resp.\ $\{-p\}$) is nontrivial in $\rH^1(G_{\Q_p}, \Z/2)$,  so case (i)  (resp.\ case (ii)) does not hold. 
When $p \equiv 3$ mod $4$, the class $\{ -p^{2m+1} \cdot p\} = \{-1\}$ is nontrivial in $\rH^1(G_{\Q_p}, \Z/2)$ and  case (iii)  does not hold.

For $p \equiv 1$ mod $4$, we have a fourth root of unity $\zeta_4 \in \Q_p$ and thus the product $-p^{2m+1}\cdot p$ is a square in $\Q_p$. In this case the first equation of (iii) is satisfied, but the second is not: 
\[
\{2\sqrt{-p^{2m+2}}\} \cup p = \{2p^{m+1}\zeta_4\} \cup p = \{2 \zeta_4\} \cup p = \{(1+\zeta_4)^2\} \cup p = 0
\]
because $p \cup p = 0$ by \ref{cup_prod_table_Q_p}.

For $(p^{2m},p)$ one shows similarly that $\delta_3^{(\operatorname{mod} 2,p)}(p^{2m},p) = 0$. Now cases (ii) and (iii) do not apply for obvious reasons and (i) can at most apply for $p \equiv 1$ mod $4$. But then
 $$ 
 \{ 2 \sqrt{-p^{2m}}\} \cup \{ p \} = (\{2\zeta_4\} + m\{p\}) \cup \{p\}  = \{(1+\zeta_4)^2 \cup p = 0
 $$ 
 vanishes as well.
\end{example}

\begin{example}\label{del32p((1-x)(-x),x)=0}
Let $p$ be a prime congruent to $3$ mod $4$. Let $x \in \Z - \{0,1\}$ be divisible by $p$. Then 
$$ 
\delta_3^{(\operatorname{mod} 2,p)}((1-x)(-x),x) = 0.
$$
Note that by Proposition~\ref{some_pts_in_Kerdelta_2}, the point $((1-x)(-x),x)$ is in the kernel of $\delta_2^{(\mdl 2,p)}$, and even in the kernel of $\delta_2$.

We again show that neither case (i)-(iii) of Theorem~\ref{Prop_ploc_mod2_3nil_obs} holds. The element $1-x$ is a square in $\Q_p$, as $1-x \equiv 1$ mod $p$, whence 
\[
\{ (1-x)(-x) \} = \{ -x \}  \in \rH^1(G_{\Q_p}, \Z/2).
\]
Since $p \equiv 3$ mod $4$, the class $\{-1\}$ is nonzero in $\rH^1(G_{\Q_p}, \Z/2)$. Therefore case (iii)  does not hold.
%
%
Case (ii) holds if and only if $\{-x\} = 0$ and
$$
\{2 \sqrt{-x} \}\cup \{(1-x)(-x)\} + \{2\} \cup \{x\} \neq 0.
$$ 
Since $(-x,x)$ is the image of a rational tangential base point by Lemma~\ref{AbelJacobi(tgtpnt)}, the obstruction $\delta_3^{(\mdl 2, p)} (-x,x) = 0$ vanishes.  By Theorem~\ref{Prop_ploc_mod2_3nil_obs} case (ii), this implies that 
$$
\{2 \sqrt{-x} \} \cup \{(-x)\} + \{2\} \cup \{x\} = 0
$$
when $\{-x\} = 0$, so (ii) does not hold for $((1-x)(-x),x)$, because $$\{(1-x)(-x)\} = \{(-x)\}.$$ 

Case (i) holds if and only if $\{(1-x)(x)\}= \{x\} = 0$ and 
$$
\{2 \sqrt{(1-x)x} \} \cup \{x\} \neq 0
$$ 
which is impossible.
\end{example}

\subsection{A global mod $2$ calculation}\label{loc_global_delta3mod2(-p^3,p)} The local calculations of~\ref{evaluation_delta32p} and~\ref{delta_3mod2R} allow us to evaluate the global obstruction $\delta_3^{\mdl 2}$ on $(-p^3,p)$ in $\Jac (\pmQ) (\Q)$. This evaluation relies on the 
Hasse--Brauer--Noether local/global principle for the ($2$-torsion) of the Brauer group, 
see \cite{coh_num_fields} Theorem 8.1.17,
\begin{equation} \label{eq:localglobalBr2tors}
0 \to \rH^2(G_{\Q}, \Z/2) \to \bigoplus_v \rH^2(G_{\Q_p},\Z/2\Z) \xrightarrow{\sum_v \inv_v} \frac{1}{2}\Z/\Z \to 0
\end{equation} but is more subtle, as the evaluation of $\delta_3^{\mdl 2}$ over $\Q$ depends on the lifts of a point of the Jacobian to $\rH^1(G_{\Q}, \pi/[\pi]_3^2)$, whereas each evaluation of $\delta_3^{(\mdl 2,p)}$ depends on the lifts to $\rH^1(G_{\Q_p}, \pi/[\pi]_3^2)$. One may not be able to find a global lift to restricting to some given set of local lifts.

In Proposition \ref{(b,(1-b)^{m^2})lifts_H1(Gk,pi/pi^m_4)}, it was shown that $$\delta_3^{\mdl 2} (b, (1-b)^4) = 0$$ for $k$ a number field and $b$ in $k- \{0,1\} = \pmk (k)$, giving a calculation of the global obstruction $\delta_3^{\mdl 2}$ on a point of the Jacobian not lying on the curve or coming from a tangential base point. However, the point $(b, (1-b)^4)$ determines the same element of $\rH^1(G_{\Q}, \pi/([\pi]_2 \pi^4))$ as the point $(b,1)$ which is the image of a rational tangential point by Lemma \ref{AbelJacobi(tgtpnt)}, so this vanishing is trivial. 

We now let $p$ vary through the primes congruent to $1$ mod $4$, and evaluate $\delta_3^{\mdl 2}$ on the family of points $(-p^3,p)$. Note that $(-p^3,p)$ does not determine the same element of $\rH^1(G_{\Q}, \pi/([\pi]_2 \pi^4))$ as a rational point or tangential point of $\pmQ$ by \eqref{aj(x)_is_(x,1-x)} and Lemma \ref{AbelJacobi(tgtpnt)}, so this gives a nontrivial calculation of $\delta_3^{\mdl 2}$ over $\Q$. 

\begin{proposition}\label{delta3mod2(-p^3,p)} Let $p$ be a prime congruent to $5$ mod $8$. Consider $(-p^3,p)$ in $\Jac(\pmQ) (\Q)$. Then $\delta_3^{\mdl 2} (-p^3,p) = 0$.\end{proposition}

\begin{proof}
We evaluate the obstruction as triple Massey products with compatible defining systems by Theorem~\ref{delta_3_Massey_prod_8_10}
\begin{equation} \label{eq:-p^3,pMassey}
\delta_{3,[[x,y],x]}^{2} (-p^3,p) = \langle p^3,-p^3,p\rangle
\end{equation}
\[
\delta_{3,[[x,y],y]}^{2} (-p^3,p) = - \langle -p,p,-p^3\rangle - \{2\} \cup p
\]
valued in $\rH^2(G_{\Q},\Z/2\Z)$, using Lemma~\ref{frakf2nilmod2} to evaluate $\overline{f}$. 

Let $S = \{2,p,\infty\}$ and let $\Q_S$ denote the maximal extension of $\Q$ unramified outside $S$. Then all classes $\{-1\}$, $\{\pm p\}$, $\{2\}$, etc. involved in \eqref{eq:-p^3,pMassey} are unramified outside $S$, i.e. they already lie in $\rH^1(\Gal(\Q_S/\Q),\Z/2\Z)$. The map 
\begin{equation}\label{QsMod2Qmod2H2}
\rH^2(\Gal(\Q_S/\Q),\Z/2\Z) \inj \rH^2(G_\Q,\Z/2\Z)
\end{equation}
is injective. One way to see this injectivity is: let $\mathcal{O}_{\Q,S}$ denote the $S$ integers of $\Q$ and let  $U = \Spec \mathcal{O}_{\Q,S}$. The \'etale cohomology groups $\rH^*(U, \Z/2\Z)$ are isomorphic to the Galois cohomology groups $\rH^*(\Gal(\Q_S/\Q), \Z/2\Z)$ by \cite[Appendix 2 Prop 3.3.1]{Haberland}, and by the Kummer exact sequence, the sequence $$\rH^1(U, \G_m) \rightarrow \rH^2(U,\Z/2\Z) \rightarrow \rH^2(U,\G_m)$$ is exact in the middle. Since $\mathcal{O}_{\Q,S}$ is a principal ideal domain, $\rH^1(U, \G_m)=0$, and by \cite[III 2.22]{MilneEC}, the natural map $\rH^2(U,\G_m) \rightarrow \rH^2(G_{\Q},\G_m)$ is an injection. Thus  \eqref{QsMod2Qmod2H2} is injective. Its image consists of the classes whose image under \eqref{eq:localglobalBr2tors} have vanishing local component except possibly at $2$,$p$ and $\infty$.  It follows we can restrict to defining systems of cochains for $\Gal(\Q_S/\Q)$. Thus the Massey product takes values in $\rH^2(\Gal(\Q_S/\Q),\Z/2\Z)$ and the local components for primes not in $S$ vanish a priori.  

We will show the vanishing of a global lift at $p$ and $\infty$, and deduce the vanishing at $2$ from reciprocity \eqref{eq:localglobalBr2tors}. 

Since $p \equiv 5$ mod $8$, we have that $2$ is not a quadratic residue mod $p$, so $\{2\}$ and $\{p\}$ span $\rH^1(G_{\Q_p}, \Z/2)$ (as in \ref{cup_prod_table_Q_p}). Thus the set of lifts $(-p^3, p)_c$ where $c$ varies among the cochains factoring through $\Gal(\Q_S/\Q)$ surjects onto the set of all lifts of $(-p^3, p)$ to $\rH^1(G_{\Q_p}, \pi/[\pi]_3^2)$. By Example \ref{compcalex}, we can therefore choose such a lift such that $\delta_3^{(2,p)} (-p^3, p)_c = 0$.

Since $p$ is congruent to $1$ mod $4$, the restriction of $\{-1\}$ to $C^1(G_{\Q_p}, \Z/2\Z)$ is $0$. By Proposition \ref{Prop_delta_3_2R}, either $\delta_3^{\mdl 2}(-p^3, p)_c = 0$ or $\delta_3^{\mdl 2}(-p^3, p)_{c + \{-1\}} = 0$, so we can choose a lift factoring through $\Gal(\Q_S/\Q)$ such that $\delta_3^{\mdl 2}$ vanishes at both $p$ and $\R$. 
\qed \end{proof}

\begin{remark}\label{Poitou-Tate_remark} 
(1) In the proof of Proposition~\ref{delta3mod2(-p^3,p)}, the vanishing of $\delta_3^{\mdl 2}(-p^3,p)$ was shown for $p \equiv 5$ mod $8$ by using the local vanishing of $\delta_3^{(\mdl 2, \nu)}(-p^3,p)$ and showing that the global lifts of $(-p^3, p)$ to $\rH^1(G_{\Q}, \pi/[\pi]_3^2)$ with ramification constrained to lie above $S = \{2,p,\infty\}$ surjected onto the product over the local lifts of $(-p^3, p)$ to $\rH^1(G_{\Q_\nu}, \pi/[\pi]_3^2)$ for the places $p$ and $\infty$.

(2) It would be desirable to relate $\delta_3^{\mdl 2} = 0$ to the simultaneous vanishing of all (or all but one) $\delta_3^{(2,\nu)}$, where $\nu$ varies over the places of a given number field $k$. For this, we would need to compare the set of restrictions to the $k_{\nu}$ of lifts of $(b,a)$ to $H^1(G_{k}, \pi/[\pi]_3^2)$ with the set of independently chosen lifts of $(b,a)$ to $H^1(G_{k_{\nu}}, \pi/[\pi]_3^2)$ for all $\nu$. In other words, we are interested in the map \begin{equation}\label{loc_glob_lift_map} H^1_{(b,a)}(G_{k}, \pi/[\pi]_3^2) \rightarrow \prod_{\nu} H^1_{(b,a)}(G_{k_{\nu}}, \pi/[\pi]_3^2) \end{equation} where $H^1_{(b,a)}(G_{k}, \pi/[\pi]_3^2)$ denotes the subset of $H^1(G_{k}, \pi/[\pi]_3^2)$ of lifts of $(b,a)$ and similarly for each $H^1_{(b,a)}(G_{k_{\nu}}, \pi/[\pi]_3^2)$. A nonabelian version of Poitou-Tate duality would give information about (\ref{loc_glob_lift_map}). \end{remark}

We can also evaluate $\delta_{3,[[x,y],x]}^2$ and $\delta_{3,[[x,y],y]}^2$ on a specific lift of $(-p^3,p)$, which is equivalent to the calculation of the Massey products $\langle p^3, -p^3 , p\rangle$ and $\langle -p, p, -p^3\rangle$ with the defining systems specified in Remark \ref{Massey_Thm_Remark} (ii), and the mod $2$ cup product $\{2\} \cup \{ p \}$. The cup product $\{2\} \cup \{ p \}$ can be calculated with \ref{cup_prod_GQ_Z/2}; it vanishes except at $2$ and $p$, and at $p$, $\{2\} \cup \{ p \}$ vanishes if and only if $p \equiv \pm 1$ mod $8$. An arbitrary defining system for $\langle p^3, -p^3, p\rangle$ or $\langle -p, p, -p^3 \rangle$ produces Massey products differing from the originals by cup products, which can also be evaluated with \ref{cup_prod_GQ_Z/2}. So evaluating $\delta_{3,[[x,y],x]}^2$ and $\delta_{3,[[x,y],y]}^2$ on a specific lift allows for the computation of $\langle p^3,-p^3, p\rangle$ and $\langle -p, p, -p^3\rangle$ in $\rH^2(G_{\Q}, \Z/2)$ with any defining system. We remark that a complete computation of the triple Massey product on $\rH^1(\Gal(k_S(2)/k), \Z/2)$ for certain maximal $2$-extensions with restricted ramification $k_S(2)$ of a number field $k$ is given in \cite[II \S 1]{Vogel_thesis}.

\begin{remark}Note that $\{-p^3 \} = \{ -p \}$ and $\{p^3 \} = \{p \}$ in $\rH^1(G_{\Q}, \Z/2)$. The reason for distinguishing between, say, $-p^3$ and $-p$ in $\langle p^3 , -p^3, p \rangle$ is that the defining systems to evaluate $\delta_{3, [[x,y],x]}^2$ and $\delta_{3, [[x,y],y]}^2$ are different for $-p^3$ and $-p$, as they depend on the image of $-p^3$ in $\rH^1(G_{\Q}, \Z/4(1))$. However, for the discussion of evaluating triple Massey products of elements of $\rH^1(G_{\Q}, \Z/2)$ for any defining system, the distinction is of course irrelevant.\end{remark}

Consider the following lift of $(-p^3,p)$: choose compatible $n^{th}$ roots of $p$, and let $\{p\}$ denote the corresponding element of $C^1(G_{\Q}, \Zhat(1))$ via the Kummer map \ref{Kummermodncocycle}. (As above, $\{p\}$ will sometimes be abbreviated by $p$.) Note that the chosen $n^{th}$ roots of $p$ give rise to a choice of compatible $n^{th}$ roots of $-p^3$ such that the corresponding element of $C^1(G_{\Q}, \Zhat(1))$ is $3(p + \frac{\chi-1}{2})$. It is therefore consistent to let $\{-p^3\}$ and $-p^3$ denote $3(p + \frac{\chi-1}{2})$. Let $c_0 = 3 {p \choose 2}$ in $C^1(G_{\Q}, \Zhat(2))$. Let $(-p^3, p)_{c_0}$ in $C^1(G_{\Q}, \pi/[\pi]_3)$ be as in Corollary \ref{lifts(b,a)_c_of(b,a)}, i.e. for all $g$ in $G_{\Q}$ $$(-p^3, p)_{c_0}(g) = y^{\{p\}(g)}x^{\{-p^3\}(g)}[x,y]^{c_0(g)},$$so $(-p^3, p)_{c_0}$ is a cocycle lifting $(-p^3, p)$. The image of $(-p^3, p)_{c_0}$ under the map $C^1(G_{\Q}, \pi/[\pi]_3) \rightarrow C^1(G_{\Q}, \pi/[\pi]_3^2)$ will also be denoted $(-p^3, p)_{c_0}$.

\begin{proposition}\label{eval_delta_3_2_spec_lift} Let $p$ be a prime congruent to $1$ mod $4$. Let $(-p^3, p)_{c_0}$ be as above. $\delta_{3}^{\mdl 2}(-p^3, p)_{c_0}$ is the element of $H^2(G_{\Q}, [\pi]_3/[\pi]_4([\pi]_3)^2)$ determined by $$ \delta_{3,[[x,y],x]}^{(2,p)} (-p^3,p)_{c_0} = \delta_{3,[[x,y],y]}^{(2,p)} (-p^3,p)_{c_0}= 2 \cup p = \begin{cases} \frac{1}{2} & \text{if $p \equiv 5 \mod 8$,}
\\
0 &\text{if $p \equiv 1 \mod 8$.}
\end{cases}$$ $$  \delta_{3,[[x,y],x]}^{(2,\nu)} (-p^3,p)_{c_0} = \delta_{3,[[x,y],y]}^{(2,\nu)} (-p^3,p)_{c_0}= 0$$ for $\nu$ equal to $\R$ or a finite odd prime not equal to $p$.\end{proposition}

Here, $\rH^2(G_{\Q_p},\Z/2\Z)$ is identified with the two torsion of $\Q/\Z$ for all finite primes $p$ via the invariant map, and elements of $H^2(G_{\Q}, \Z/2)$ are identified with their images under (\ref{eq:localglobalBr2tors}) .

\begin{proof}
Let $S = \{2,p,\infty\}$ and let $\Q_S$ denote the maximal extension of $\Q$ unramified outside $S$. The cocycle $(-p^3,p)_{c_0}$ factors through $\Gal(\Q_S/\Q)$, and it follows that $$\delta_{3,[[x,y],x]}^{(2,\nu)} (-p^3,p)_{c_0} = \delta_{3,[[x,y],y]}^{(2,\nu)} (-p^3,p)_{c_0}= 0$$ for $\nu$ equal to any prime not in $S$. 

The obstruction $\delta_{3}^{(2,\nu)} (-p^3,p)_{c_0}$ for $\nu = \R$ decomposes into two elements $\delta_{3,[[x,y],x]}^{(2, \R)}(-p^3,p)_{c_0}$ and $\delta_{3,[[x,y],y]}^{(2, \R)}(-p^3,p)_{c_0}$ of $\Z/2 \cong H^2(G_{\R},\Z/2)$, obtained by evaluating each of the cocycles given in Proposition \ref{delta_3_p1minus3_cocycle} at $(g_1, g_2) = (\tau,\tau)$, where $\tau$ denote complex conjugation in $G_{\Q}$ c.f. \ref{H2(G_R,Z/2)}. Note that since $p$ is positive, the equalities $\{-p^3 \}(\tau) = 1$, $\{p\}(\tau)=0$, $ {-\{p\}(\tau) + 1\choose 2} = 0$, and $c_0(\tau) =0$ hold in $\Z/2$. Substituting these equations into the cocycles in Proposition \ref{delta_3_p1minus3_cocycle} shows that $\delta_{3,[[x,y],x]}^{(2, \R)}(-p^3,p)_{c_0} = 0 $ and $\delta_{3,[[x,y],y]}^{(2, \R)}(-p^3,p)_{c_0} = 0$.

By Lemma \ref{rank<2} case (\ref{loc_obst_b=a}), we have that \begin{equation}\label{sump-p^3atp}\delta_{3,[[x,y],x]}^{(2,p)} (-p^3,p)_c + \delta_{3,[[x,y],y]}^{(2,p)} (-p^3,p)_c \end{equation} does not depend on the choice of lift. By Example \ref{compcalex}, there is a lift $(-p^3,p)_c$ such that $$\delta_{3,[[x,y],x]}^{(2,p)} (-p^3,p)_c = \delta_{3,[[x,y],y]}^{(2,p)} (-p^3,p)_c =0$$ and it follows that (\ref{sump-p^3atp}) vanishes.

It follows from Proposition \ref{delta_3_p1minus3_cocycle} that \begin{equation}\label{delta3[[x,y],x]2(-p^3,p)c_0_explicit}\delta_{3,[[x,y],x]}^{(2,p)} (-p^3,p)_{c_0} = ( {p \choose 2}  + { -(p + \frac{\chi -1}{2}) \choose 2} ) \cup p.\end{equation} To see this, note that $ \frac{\chi -1}{2} = 0$ and $-p^3 = p$ in $C^1(G_{\Q_p} ,\Z/2)$. Thus the cocycle $g \mapsto p(g) 3(p + \frac{\chi -1}{2}) (g)$ equals the cocycle $g \mapsto p(g)$ in $C^1(G_{\Q_p} ,\Z/2)$. (Note that equating these two cocycles requires identifying $\Z/2(1)$ with $\Z/2(2)$, so the weight is not being respected!) Substituting these equalities into Proposition \ref{delta_3_p1minus3_cocycle} implies $$\delta_{3,[[x,y],x]}^{(2,p)} (-p^3,p)_{c_0} = {p \choose 2} \cup p + { 3(p + \frac{\chi -1}{2}) + 1 \choose 2} \cup p + p \cup p.$$ Then note that in $C^1(G_{\Q_p} ,\Z/2)$, $${ 3(p + \frac{\chi -1}{2}) + 1 \choose 2} = { 3(p + \frac{\chi -1}{2}) \choose 2} + 3(p + \frac{\chi -1}{2}) = { -(p + \frac{\chi -1}{2}) \choose 2} + p,$$ showing (\ref{delta3[[x,y],x]2(-p^3,p)c_0_explicit}). (It is important to distinguish between $3(p + \frac{\chi -1}{2})$ and $p$ in the binomial coefficient as these two cocycles are not equal in  $C^1(G_{\Q_p} ,\Z/4(1))$.)

For any two elements $d_1,d_2$ in $\Z/4$, direct calculation shows $${d_1 + d_2 \choose 2} - {d_1 \choose 2} - {d_2 \choose 2} = d_1 d_2$$ in $\Z/2$. Thus $$ {p \choose 2}  + { -(p + \frac{\chi -1}{2}) \choose 2}  = {-\frac{\chi -1}{2} \choose 2} + p (p + \frac{\chi -1}{2}) = {-\frac{\chi -1}{2} \choose 2} + p.$$ Combining with the above, we see  $$ \delta_{3,[[x,y],x]}^2 (-p^3,p)_{c_0} = {-\frac{\chi -1}{2} \choose 2} \cup p.$$ Since $p$ is congruent to $1$ mod $4$, $\Q_p$ contains a primitive fourth root of unity and $\chi(g) \equiv 1 \mod 4$ for every $g$ in $G_{\Q_p}$. Therefore, $\frac{\chi(g) -1}{2}$ is $0$ or $2$ mod $4$, whence ${-\frac{\chi -1}{2} \choose 2}$ is $0$ for $g$ fixing the eight roots of unity and $1$ otherwise. It follows that $$ \delta_{3,[[x,y],x]}^2 (-p^3,p)_{c_0} = \begin{cases} \frac{1}{2} & \text{if $p \equiv 5 \mod 8$,}
\\
0 &\text{if $p \equiv 1 \mod 8$.}
\end{cases}$$
\qed \end{proof}

\bibliographystyle{PIA}


\bibliography{PIA}


\end{document}